\documentclass[11pt]{amsart}

\usepackage{colonequals}
\usepackage[autostyle, english = american]{csquotes}
\usepackage{dirtytalk}
\usepackage{makecell}
\usepackage{subcaption}
\usepackage{subfiles}
\usepackage{comment}
\usepackage{float}
\usepackage{marginnote}
\usepackage{tabu}
\usepackage{euscript}
\usepackage{mathdots}
\usepackage[dvipsnames]{xcolor}
\usepackage{graphicx}			
\usepackage{amssymb}
\usepackage{mathrsfs}
\usepackage{amsthm}
\usepackage{amsmath}
\usepackage{stmaryrd}
\usepackage{tikz}
\usepackage{tikz-cd}
\usepackage{accents}
\usepackage{upgreek}
\usepackage{enumerate}
\usepackage{bm}
\usepackage{mathtools}
\usetikzlibrary{patterns}
\usepackage[all]{xy}
\usepackage{caption}
\usepackage{url}
\usepackage{float}
\usepackage{todonotes} 
\usepackage{colonequals}
\usepackage{bbm}
\usepackage{longtable}
\usepackage[full]{textcomp}
\usepackage[cal=cm]{mathalfa}
\usepackage{xparse}
\usepackage{comment}
\usepackage{cite}
\usepackage{enumitem}

\usetikzlibrary{calc}
\usetikzlibrary{fadings}
\usetikzlibrary{decorations.pathmorphing}
\usetikzlibrary{decorations.pathreplacing}
\usepackage{tikz,tikz-cd,tikz-3dplot}
\usepackage{pgfplots}
\usetikzlibrary{arrows,shadows,positioning, calc, decorations.markings, 
	hobby,quotes,angles,decorations.pathreplacing,intersections,shapes}
\usetikzlibrary{fillbetween,backgrounds}

\newcounter{step}
\newcommand{\proofstep}[1]{\refstepcounter{step}\noindent\emph{\underline{\smash{Step~\thestep: #1.}}}}

\usepackage[margin=0.9in]{geometry}

\tikzset{
	commutative diagrams/.cd, 
	arrow style=tikz, 
	diagrams={>=stealth}
}
\tikzset{
	arrow/.pic={\path[tips,every arrow/.try,->,>=#1] (0,0) -- +(0,4pt);},
	pics/arrow/.default={triangle 90}
}
\tikzset{->-/.style={decoration={
			markings,
			mark=at position .6 with {\arrow{latex}}},postaction={decorate}}
}
\tikzset{
	c/.style={every coordinate/.try}
}

\theoremstyle{plain}
\newtheorem{innercustomthm}{Theorem}
\newenvironment{customthm}[1]
{\renewcommand\theinnercustomthm{#1}\innercustomthm}
{\endinnercustomthm}

\makeatletter
\def\@tocline#1#2#3#4#5#6#7{\relax
	\ifnum #1>\c@tocdepth 
	\else
	\par \addpenalty\@secpenalty\addvspace{#2}%
	\begingroup \hyphenpenalty\@M
	\@ifempty{#4}{%
		\@tempdima\csname r@tocindent\number#1\endcsname\relax
	}{%
		\@tempdima#4\relax
	}%
	\parindent\z@ \leftskip#3\relax \advance\leftskip\@tempdima\relax
	\rightskip\@pnumwidth plus4em \parfillskip-\@pnumwidth
	#5\leavevmode\hskip-\@tempdima
	\ifcase #1
	\or\or \hskip 1em \or \hskip 2em \else \hskip 3em \fi%
	#6\nobreak\relax
	\dotfill\hbox to\@pnumwidth{\@tocpagenum{#7}}\par
	\nobreak
	\endgroup
	\fi}
\makeatother

\newcounter{marginnote}
\DeclareMathAlphabet{\mathpzc}{OT1}{pzc}{m}{it}

\usepackage[backref=page]{hyperref}
\hypersetup{
  colorlinks   = true,          
  urlcolor     = PineGreen,          
  linkcolor    = PineGreen,          
  citecolor   = PineGreen             
}

\usepackage{cleveref}

\theoremstyle{plain}
\newtheorem{theorem}{Theorem}[section]
\newtheorem{thm}{Theorem}[section]

\newtheorem{conjecture}[theorem]{Conjecture}

\newtheorem{lemma}[theorem]{Lemma}
\newtheorem{proposition}[theorem]{Proposition}

\theoremstyle{definition}

\newtheorem{remark}[theorem]{Remark}

\newtheorem*{runningexample*}{Running example}

\newtheorem*{aside*}{Aside}

\newtheorem{definition}[theorem]{Definition}
\newtheorem{example}[theorem]{Example}

\newtheorem{proposition-definition}[theorem]{Proposition-Definition}

\newcommand{\xdashleftrightarrow}[2][]{\ext@arrow 3359\leftrightarrowfill@@{#1}{#2}}

\newcommand{\Gm}{\mathbb{G}_{\operatorname{m}}}

\newcommand{\ol}[1]{\overline{#1}}
\newcommand{\ul}[1]{\underline{#1}}
\newcommand{\bcd}{\begin{center}\begin{tikzcd}}
		\newcommand{\ecd}{\end{tikzcd}\end{center}}

\newcommand{\Aaff}{\mathbb{A}}

\newcommand{\PP}{\mathbb{P}}
\newcommand{\OO}{\mathcal{O}}
\newcommand{\N}{\mathbb{N}}
\newcommand{\Z}{\mathbb{Z}}

\newcommand{\kfield}{\Bbbk}

\newcommand{\Mcal}{\mathcal{M}}

\newcommand{\Gcal}{\mathcal{G}}
\newcommand{\Ucal}{\mathcal{U}}

\newcommand{\pt}{\mathrm{pt}}

\newcommand{\Ugen}{\Ucal}

\newcommand{\sgn}{\operatorname{sgn}}

\newcommand{\Ker}{\operatorname{Ker}}

\crefname{equation}{eq.}{eqs.}
\crefname{eqnarray}{eq.}{eqs.}
\crefname{conjecture}{conjecture}{conjectures}
\crefname{lemma}{lemma}{lemmas}
\crefname{theorem}{theorem}{theorems}
\crefname{claim}{claim}{claims}
\crefname{remark}{remark}{remarks}
\crefname{proposition}{proposition}{propositions}
\crefname{section}{section}{sections}
\crefname{appendix}{appendix}{appendices}
\crefname{corollary}{corollary}{corollaries}
\crefname{figure}{figure}{figures}
\crefname{table}{table}{tables}
\crefname{example}{example}{examples}
\crefname{assumption}{assumption}{assumptions}
\crefname{definition}{definition}{definitions}
\crefname{innercustomthm}{theorem}{theorems}
\crefname{innercustomconj}{conjecture}{conjectures}

\setlist[enumerate,1]{label=(\arabic*),itemsep=0.9ex}
\setlist[itemize]{itemsep=0.9ex}

\usetikzlibrary{decorations.pathmorphing,shapes,calc,decorations.markings}

\renewcommand{\tilde}[1]{\widetilde{#1}}

\newcommand{\Ical}{\mathcal{I}}

\newcommand{\Id}{\mathrm{Id}}

\DeclareMathOperator{\codim}{codim}

\DeclareMathOperator{\Gr}{Gr}
\DeclareMathOperator{\PGL}{PGL}

\DeclareMathOperator{\pr}{pr}
\DeclareMathOperator{\rk}{rk}
\DeclareMathOperator{\len}{len}
\renewcommand{\O}{\mathcal{O}}

\renewcommand{\P}{\mathbb{P}}
\newcommand{\M}{\mathcal{M}}

\newcommand{\abs}[1]{\left\lvert#1\right\rvert}

\newcommand{\fgt}{\operatorname{fgt}}

\newcommand{\rob}[1]{{\color{teal}[Rob:] #1 }}
\newcommand{\navid}[1]{{\color{olive}[Navid:] #1 }}

\begin{document}

\title[Projective configuration counts]{Counting point configurations in projective space}
\author{Alex Fink, Navid Nabijou, Rob Silversmith}

\begin{abstract} We investigate the enumerative geometry of point configurations in projective space. We define ``projective configuration counts'': these enumerate configurations of points in projective space such that certain specified subsets are in fixed relative positions. The $\P^1$ case recovers cross-ratio degrees, which arise naturally in numerous contexts.
We establish two main results. The first is a combinatorial upper bound given by the number of weighted transversals of a bipartite graph. The second is a recursion that relates counts associated to projective spaces of different dimensions, by projecting away from a given point. 
Key inputs include the Gelfand--MacPherson correspondence, the Jacobi--Trudi and Thom--Porteous formulae, and the notion of surplus from matching theory of bipartite graphs.
\end{abstract}

\maketitle
\tableofcontents

\section*{Introduction}

\subsection{Cross-ratio degrees} Consider the moduli space $\Mcal_{0,n}$ of configurations of $n$ distinct labelled points in $\P^1$ up to the action of $\PGL_2$. A \textbf{cross-ratio degree} counts the number of such configurations satisfying $n-3$ given cross-ratio constraints. Formally, it is the degree of a product of forgetful maps
\[ \Mcal_{0,n} \to \prod_{j=1}^{n-3} \Mcal_{0,I_j} \]
arising from a choice of $4$-element subsets $I_1,\ldots,I_{n-3} \subseteq [n] \colonequals \{1,\ldots,n\}$. Cross-ratio degrees have been studied from numerous perspectives, enjoying connections to hypertree projections \cite{CastravetTevelev2013}, graph matchings and linear Gromov--Witten theory \cite{Silversmith2022}, scattering amplitudes \cite{Silversmith2024}, tropical curves \cite{Goldner2021}, graph rigidity \cite{GalletGraseggerSchicho2020}, and intersection theory on the compactification $\overline{\Mcal}_{0,n}$ \cite{BrakensiekEurLarsonLi2023}.

\subsection{Projective configuration counts} We initiate the study of cross-ratio degrees in higher dimensions. The analogue of $\Mcal_{0,n}$ is the moduli space
\[ X(r,[n]) \]
of configurations of $n$ linearly general labelled points in $\P^{r-1}$ up to the action of $\PGL_r$. Since $\PGL_r$ acts on $\P^{r-1}$ simply $(r+1)$-transitively, we assume $n \geqslant r+1$. When $n=r+1$ the moduli space is a point, and when $n=r+2$ the moduli space is isomorphic to the complement in $\P^{r-1}$ of a union of linear subspaces.

Fix therefore $r \geqslant 2$, $n \geqslant r+1$ and let $k \colonequals n-r-1$. Fix subsets 
\[ I_1,\ldots,I_k \subseteq [n] \]
of size $r+2$, and let $\Ical = (I_1,\ldots,I_k)$. Consider the product of forgetful maps
\[ \fgt_\Ical \colon X(r,[n]) \to \prod_{j=1}^{k} X(r,I_j). \]
The numerics are chosen such that this is a morphism between smooth, connected varieties of the same dimension. The associated \textbf{projective configuration count} is denoted and defined:
\[ d_{r,n}(\Ical) \colonequals \deg (\fgt_\Ical) \in \N. \]
In this paper we establish two main results on projective configuration counts: a combinatorial upper bound (Theorem \ref{thm: upper bound introduction}) and a dimension reduction principle (Theorem \ref{thm: dimension reduction introduction}). 

\subsection{Upper bound} To state our first main theorem, we introduce some terminology, illustrated in Example \ref{example naive upper bound fails} below. We encode the datum $\Ical$ as a bipartite graph which we call the \textbf{configuration graph} $\Gamma(\Ical)$, with left vertices indexed by the markings $i \in [n]$ and right vertices indexed by the subsets $I_j$ for $j \in [k]$. There is an edge connecting $i \in [n]$ and $j \in [k]$ if and only if $i \in I_j$. 

We have $k<n$, but if we choose a subset $S \subseteq [n]$ of size $r+1$ then we can construct the \textbf{pruned configuration graph}
$\Gamma(\Ical)\!\setminus\!S$ by deleting the vertices indexed by $S$ (and all adjacent edges). Note that $\Gamma(\Ical)\!\setminus\!S$ has the same number of left and right vertices.

Finally, given an integer $m \geqslant 1$ an \textbf{$m$-weighted transversal} of $\Gamma(\Ical)\!\setminus\!S$ is an assignment of non-negative integer weights to the edges of $\Gamma(\Ical)\!\setminus\!S$, such that at each vertex the sum of the adjacent weights is equal to $m$. A $1$-weighted transversal is otherwise known as a \textbf{perfect matching} or \textbf{transversal}.

Our first main result is an upper bound for projective configuration counts, generalising \cite[Theorem~1.1]{Silversmith2022}.
\begin{customthm}{A}[Theorem \ref{thm: upper bound}] \label{thm: upper bound introduction} The projective configuration count $d_{r,n}(\Ical)$ is at most the number of $(r-1)$-weighted transversals of $\Gamma(\Ical)\!\setminus\!S$.	
\end{customthm}

\begin{example} \label{example naive upper bound fails} Take $r=3$ and $n=8$, so that $k=4$. Consider the following data $\Ical=(I_1,I_2,I_3,I_4)$ for a projective configuration count:
\begin{align*}
I_1 & = \{1,2,3,4,5\}, \\
I_2 & = \{3,4,5,6,7\}, \\
I_3 &  = \{5,6,7,8,1\}, \\
I_4 & = \{7,8,1,2,3\}.
\end{align*}
Let $S=\{1,3,5,7\}$. The configuration graph and pruned configuration graph are:
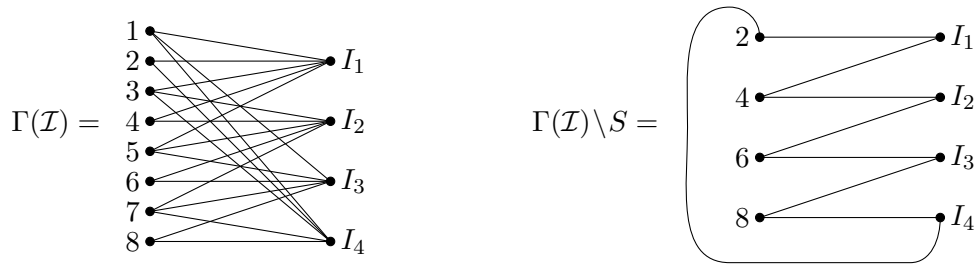
\begin{align*}
\Gamma(\Ical)&=\raisebox{-50pt}{
    \begin{tikzpicture}[scale=0.8]
    \draw[fill=black] (0,0) circle[radius=2pt];
    \draw (0,0) node[left]{$2$};
    \draw[fill=black] (0,-1) circle[radius=2pt];
     \draw (0,-1) node[left]{$4$};
      \draw[fill=black] (0,-2) circle[radius=2pt];
     \draw (0,-2) node[left]{$6$};
      \draw[fill=black] (0,-3) circle[radius=2pt];
     \draw (0,-3) node[left]{$8$};
     \draw[fill=black] (0,0.5) circle[radius=2pt];
    \draw (0,0.5) node[left]{$1$};
    \draw[fill=black] (0,-.5) circle[radius=2pt];
     \draw (0,-.5) node[left]{$3$};
      \draw[fill=black] (0,-1.5) circle[radius=2pt];
     \draw (0,-1.5) node[left]{$5$};
      \draw[fill=black] (0,-2.5) circle[radius=2pt];
     \draw (0,-2.5) node[left]{$7$};
%
    \draw[fill=black] (3,0) circle[radius=2pt];
    \draw (3,0) node[right]{$I_1$};
    \draw[fill=black] (3,-1) circle[radius=2pt];
    \draw (3,-1) node[right]{$I_2$};
    \draw[fill=black] (3,-2) circle[radius=2pt];
    \draw (3,-2) node[right]{$I_3$};
    \draw[fill=black] (3,-3) circle[radius=2pt];
    \draw (3,-3) node[right]{$I_4$};
%
    \draw (0,0.5)--(3,0);
    \draw (0,0.5)--(3,-2);
    \draw (0,0.5)--(3,-3);
    \draw (0,0)--(3,0);
    \draw (0,0)--(3,-3);
    \draw (0,-0.5)--(3,0);
    \draw (0,-0.5)--(3,-1);
    \draw (0,-0.5)--(3,-3);
    \draw (0,-1)--(3,0);
    \draw (0,-1)--(3,-1);
    \draw (0,-1.5)--(3,0);
    \draw (0,-1.5)--(3,-1);
    \draw (0,-1.5)--(3,-2);
    \draw (0,-2)--(3,-1);
    \draw (0,-2)--(3,-2);
    \draw (0,-2.5)--(3,-1);
    \draw (0,-2.5)--(3,-2);
    \draw (0,-2.5)--(3,-3);
    \draw (0,-3)--(3,-2);
    \draw (0,-3)--(3,-3);
\end{tikzpicture}
}
    &\Gamma(\Ical)\!\setminus\!S&=\raisebox{-50pt}{
    \begin{tikzpicture}[scale=0.8]
    \draw[fill=black] (0,0) circle[radius=2pt];
    \draw (0,0) node[left]{$2$};
    \draw[fill=black] (0,-1) circle[radius=2pt];
     \draw (0,-1) node[left]{$4$};
      \draw[fill=black] (0,-2) circle[radius=2pt];
     \draw (0,-2) node[left]{$6$};
      \draw[fill=black] (0,-3) circle[radius=2pt];
     \draw (0,-3) node[left]{$8$};
%
    \draw[fill=black] (3,0) circle[radius=2pt];
    \draw (3,0) node[right]{$I_1$};
    \draw[fill=black] (3,-1) circle[radius=2pt];
    \draw (3,-1) node[right]{$I_2$};
    \draw[fill=black] (3,-2) circle[radius=2pt];
    \draw (3,-2) node[right]{$I_3$};
    \draw[fill=black] (3,-3) circle[radius=2pt];
    \draw (3,-3) node[right]{$I_4$};
%
    \draw (0,0) -- (3,0);
    \draw (0,0) to [out=90,in=0] (-0.5,0.5) to [out=180,in=90] (-1.2,-1.5) to [out=270,in=180] (-0.5,-3.75) to [out=0,in=180] (2.5,-3.75) to [out=0,in=270] (3,-3);
    \draw (0,-1) -- (3,0);
    \draw (0,-1) -- (3,-1);
    \draw (0,-2) -- (3,-1);
    \draw (0,-2) -- (3,-2);
    \draw (0,-3) -- (3,-2);
    \draw (0,-3) -- (3,-3);
\end{tikzpicture}
}
\end{align*}

The graph $\Gamma(\Ical)\!\setminus\!S$ has exactly three $2$-weighted transversals, namely:
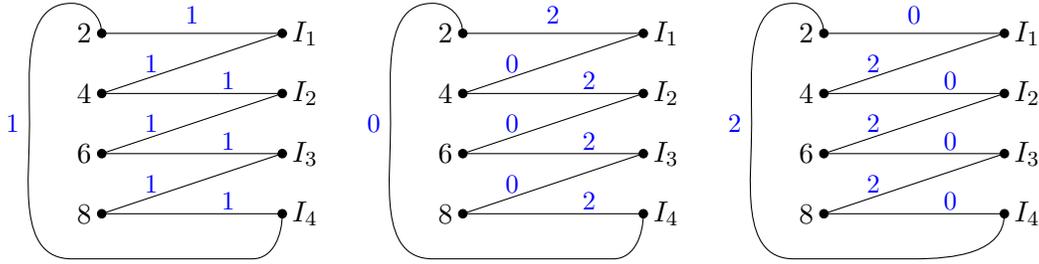
\[
\begin{tikzpicture}[scale=0.8]
    \draw[fill=black] (0,0) circle[radius=2pt];
    \draw (0,0) node[left]{$2$};

    \draw[fill=black] (0,-1) circle[radius=2pt];
     \draw (0,-1) node[left]{$4$};
     
      \draw[fill=black] (0,-2) circle[radius=2pt];
     \draw (0,-2) node[left]{$6$};
     
      \draw[fill=black] (0,-3) circle[radius=2pt];
     \draw (0,-3) node[left]{$8$};

    \draw[fill=black] (3,0) circle[radius=2pt];
    \draw (3,0) node[right]{$I_1$};
    
    \draw[fill=black] (3,-1) circle[radius=2pt];
    \draw (3,-1) node[right]{$I_2$};

    \draw[fill=black] (3,-2) circle[radius=2pt];
    \draw (3,-2) node[right]{$I_3$};

    \draw[fill=black] (3,-3) circle[radius=2pt];
    \draw (3,-3) node[right]{$I_4$};

    \draw (0,0) -- (3,0); 
    \draw[blue] (1.5,0) node[above]{\small$1$};
    \draw (0,0) to [out=90,in=0] (-0.5,0.5) to [out=180,in=90] (-1.2,-1.5) to [out=270,in=180] (-0.5,-3.75) to [out=0,in=180] (2.5,-3.75) to [out=0,in=270] (3,-3);
    \draw[blue] (-1.2,-1.5) node[left]{\small$1$};
    \draw (0,-1) -- (3,0);
    \draw[blue] (1.1,-0.5) node[left]{\small$1$};
    \draw (0,-1) -- (3,-1);
    \draw[blue] (2.1,-1.1) node[above]{\small$1$};
    \draw (0,-2) -- (3,-1);
    \draw[blue] (1.1,-1.5) node[left]{\small$1$};
    \draw (0,-2) -- (3,-2);
    \draw[blue] (2.1,-2.1) node[above]{\small$1$};
    \draw (0,-3) -- (3,-2);
    \draw[blue] (1.1,-2.5) node[left]{\small$1$};
    \draw (0,-3) -- (3,-3);
    \draw[blue] (2.1,-3.1) node[above]{\small$1$};

    \draw[fill=black] (6,0) circle[radius=2pt];
    \draw (6,0) node[left]{$2$};

    \draw[fill=black] (6,-1) circle[radius=2pt];
     \draw (6,-1) node[left]{$4$};
     
      \draw[fill=black] (6,-2) circle[radius=2pt];
     \draw (6,-2) node[left]{$6$};
     
      \draw[fill=black] (6,-3) circle[radius=2pt];
     \draw (6,-3) node[left]{$8$};

    \draw[fill=black] (9,0) circle[radius=2pt];
    \draw (9,0) node[right]{$I_1$};
    
    \draw[fill=black] (9,-1) circle[radius=2pt];
    \draw (9,-1) node[right]{$I_2$};

    \draw[fill=black] (9,-2) circle[radius=2pt];
    \draw (9,-2) node[right]{$I_3$};

    \draw[fill=black] (9,-3) circle[radius=2pt];
    \draw (9,-3) node[right]{$I_4$};

    \draw (6,0) -- (9,0); 
    \draw[blue] (7.5,0) node[above]{\small$2$};
    \draw (6,0) to [out=90,in=0] (5.5,0.5) to [out=180,in=90] (4.8,-1.5) to [out=270,in=180] (5.5,-3.75) to [out=0,in=180] (8.5,-3.75) to [out=0,in=270] (9,-3);
    \draw[blue] (4.8,-1.5) node[left]{\small$0$};
    \draw (6,-1) -- (9,0);
    \draw[blue] (7.1,-0.5) node[left]{\small$0$};
    \draw (6,-1) -- (9,-1);
    \draw[blue] (8.1,-1.1) node[above]{\small$2$};
    \draw (6,-2) -- (9,-1);
    \draw[blue] (7.1,-1.5) node[left]{\small$0$};
    \draw (6,-2) -- (9,-2);
    \draw[blue] (8.1,-2.1) node[above]{\small$2$};
    \draw (6,-3) -- (9,-2);
    \draw[blue] (7.1,-2.5) node[left]{\small$0$};
    \draw (6,-3) -- (9,-3);
    \draw[blue] (8.1,-3.1) node[above]{\small$2$};

    \draw[fill=black] (12,0) circle[radius=2pt];
    \draw (12,0) node[left]{$2$};

    \draw[fill=black] (12,-1) circle[radius=2pt];
     \draw (12,-1) node[left]{$4$};
     
      \draw[fill=black] (12,-2) circle[radius=2pt];
     \draw (12,-2) node[left]{$6$};
     
      \draw[fill=black] (12,-3) circle[radius=2pt];
     \draw (12,-3) node[left]{$8$};

    \draw[fill=black] (15,0) circle[radius=2pt];
    \draw (15,0) node[right]{$I_1$};
    
    \draw[fill=black] (15,-1) circle[radius=2pt];
    \draw (15,-1) node[right]{$I_2$};

    \draw[fill=black] (15,-2) circle[radius=2pt];
    \draw (15,-2) node[right]{$I_3$};

    \draw[fill=black] (15,-3) circle[radius=2pt];
    \draw (15,-3) node[right]{$I_4$};

    \draw (12,0) -- (15,0); 
    \draw[blue] (13.5,0) node[above]{\small$0$};
    \draw (12,0) to [out=90,in=0] (11.5,0.5) to [out=180,in=90] (10.8,-1.5) to [out=270,in=180] (11.5,-3.75) to [out=0,in=180] (13.5,-3.75) to [out=0,in=270] (15,-3);
    \draw[blue] (10.8,-1.5) node[left]{\small$2$};
    \draw (12,-1) -- (15,0);
    \draw[blue] (13.1,-0.5) node[left]{\small$2$};
    \draw (12,-1) -- (15,-1);
    \draw[blue] (14.1,-1.1) node[above]{\small$0$};
    \draw (12,-2) -- (15,-1);
    \draw[blue] (13.1,-1.5) node[left]{\small$2$};
    \draw (12,-2) -- (15,-2);
    \draw[blue] (14.1,-2.1) node[above]{\small$0$};
    \draw (12,-3) -- (15,-2);
    \draw[blue] (13.1,-2.5) node[left]{\small$2$};
    \draw (12,-3) -- (15,-3);
    \draw[blue] (14.1,-3.1) node[above]{\small$0$};
    
\end{tikzpicture}
\]
A stochastic calculation (see Section \ref{sec: tables}) gives the projective configuration count as $d_{3,8}(\Ical) = 3$. Since the number of $2$-weighted transversals is $3$, we find that the upper bound is tight in this case (assuming the stochastic calculation is valid).
\end{example}

\begin{remark} \label{rem: need weighted transversals}
    Interestingly, Example \ref{example naive upper bound fails} also shows that the most obvious generalisation of the upper bound in \cite[Theorem~1]{Silversmith2022} --- namely the number of perfect matchings of $\Gamma(\Ical)\!\setminus\!S$ --- fails. The graph $\Gamma(\Ical)\!\setminus\!S$ has only $2$ perfect matchings, and $2<d_{3,8}(\Ical) = 3.$
\end{remark}



\subsection{Dimension reduction} We now consider the situation in which there exists a marking $i \in [n]$ which belongs to every $I_j$. In this case, we can equate the projective configuration count with a projective configuration count one dimension lower:

\begin{customthm}{B}[Theorem \ref{thm: dimension reduction}] \label{thm: dimension reduction introduction} Fix a tuple $\mathcal{I}=(I_1,\ldots,I_k)$ and append an additional marking to each constraint:
    \[ \mathcal{I}' \colonequals (I_1\sqcup\{n\!+\!1\},\ldots,I_k\sqcup\{n\!+\!1\}). \]
    Then we have $d_{r+1,n+1}(\mathcal{I}')=d_{r,n}(\mathcal{I})$.
\end{customthm}
Theorem \ref{thm: dimension reduction introduction} is proven by considering the projection morphism
\[ X(r\!+\!1,[n\!+\!1]) \to X(r,[n]) \]
corresponding to the projection $\P^r \dashrightarrow \P^{r-1}$ away from the final marking. Theorems~\ref{thm: upper bound introduction} and \ref{thm: dimension reduction introduction} can also be combined: see Section \ref{sec: combining theorems}.

\subsection{Compactifications} In the classical ($r=2$) case, the existence of a modular compactification with a recursive normal crossings boundary
\[ \Mcal_{0,n} \subseteq \ol{\Mcal}_{0,n} \]
allows one to push the intersection problem into the boundary and induct. This is the basic strategy at the heart of most major results in the subject \cite{Silversmith2022,BrakensiekEurLarsonLi2023}.

For $r \geqslant 3$ compactifications of $X(r,[n])$ have been studied extensively \cite{GerritzenPiwek1991,Kapranov1993Chow, KeelTevelev,CastravetTevelev2013,AlexeevBook,SchafflerTevelev2022}. Numerous interrelated compactifications exist, but unfortunately none has a recursive normal crossings boundary. We interpret this defect as arising from the pathologies of matroid realisation spaces: when $r \geqslant 3$, such spaces can be singular or disconnected, and the closure of one is not equal to a union of others \cite{Mnev1,Mnev2,CoreyLuber}. Because of this, the analogue of Kapranov's iterated blowup construction \cite{Kapranov1993Chow} necessarily introduces wild singularities.

The lack of a nice modular compactification complicates the subject considerably. For instance, we do not know a combinatorial algorithm to compute projective configuration counts for $r \geqslant 3$ (though there is a stochastic method: see Section \ref{sec: tables}). Constructing a nice modular compactification of $X(r,[n])$ seems the paramount problem in the field; however, we are sceptical that such a compactification exists in general.

\subsection{Prospects} Beyond questions surrounding compactifications, there are several other natural directions to pursue.
	
\begin{enumerate}
\item \textbf{Algorithm.} By studying tropicalisations of arcs, one can pursue an algorithm for computing projective configuration counts via the tropical geometry of the Bruhat--Tits building of $\PGL_r$. This would generalise Goldner's algorithm, which applies when $r=2$ \cite{Goldner2021} (as elaborated in \cite[Section~1.2.4]{Silversmith2022} and reinterpreted via tropical moduli in \cite{SilversmithFirework}). It is inspired by work \cite{SchafflerTevelev2022,GerritzenPiwek1991} on compactifying $X(r,[n])$ via Mustafin varieties, and the relationship of the latter to Bruhat--Tits buildings \cite{CHSW, ChenGibneyKrashen, HahnLi}.

\item \textbf{Positivity.} In Conjecture \ref{conj: positivity} we propose a necessary and sufficient criterion for a projective configuration count to be nonzero. This generalises the $r=2$ surplus condition of \cite[Question~4.2]{Silversmith2022}. It can also be viewed as determining the bases of an ``algebraic polymatroid'' associated to a natural embedding of $X(r,[n])$.\smallskip

\item \textbf{Psi classes.} It is possible to study psi classes on compactifications of $X(r,[n])$ and relate their degrees to projective configuration counts. We have informally investigated these on the Schaffler--Tevelev compactification \cite{SchafflerTevelev2022}. This compactification admits a morphism to a product of projective spaces, and the pullbacks of hyperplane classes can be explicitly compared to psi classes. The comparison is intricate, and thus the same is true of the relationship between the projective configuration counts and psi class integrals.

\item \textbf{Upper bounds.} The strategy used to prove the upper bound should apply in other contexts, whenever there is an open moduli space which compactifies to a product of projective spaces. In \cite{SimmsUpperBound} this is used to establish an upper bound for logarithmic Gromov--Witten invariants.
\end{enumerate}

\subsection{Sample calculations} \label{sec: tables}

\noindent Projective configuration counts can be calculated numerically, by randomly generating constraints and counting the solutions to the resulting system of equations. We implement this in accompanying \texttt{Mathematica} code.

This method is stochastic and thus not error-proof: there is a small chance that the randomly-generated constraints will fail to be generic. However, if we run the code multiple times and obtain the same answer each time, we can be reasonably confident that the calculation is correct.

The following table collects some $r=3$ projective configuration counts calculated in this way, compared against the upper bound given by weighted transversals (Theorem \ref{thm: upper bound introduction}).

\setlength\extrarowheight{3pt}
\begin{center}
\begin{tabular}{| c | c | c |}
	\hline
	$\Ical$ & $d_{r,n}(\Ical)$ & $|T_{r-1}(\Gamma(\Ical)\!\setminus\!S)|$ \\
	\hline \hline
	$12345, 23456$ & $1$ & $\mathbf{1}, 3$ \\
	\hline
	\makecell{$12347, 34567, 12567$} & $2$ & $\mathbf{3},6$ \\
	\hline
	\makecell{$12345,34567,56781,78123$} & $3$ & $\mathbf{3}, 6, 10, 14, 15, 20, 42$ \\
	\hline
        \makecell{$12345,12367,14578,14689,34569$} & $4$ & $\mathbf{6}, 10, 14, 15, 20, 21, \ldots, 187$ \\
    \hline
\end{tabular}
\end{center}
The first column records the tuple $\Ical=(I_1,\ldots,I_k)$ where we write e.g.\ $12345$ for the set $I_j=\{1,2,3,4,5\}$. The value of $n$ can be deduced from the length $k$ of $\Ical$, since $n=k+r+1=k+4$. The second column records the count $d_{r,n}(\Ical)$. The third column records the upper bound $|T_{r-1}(\Gamma(\Ical)\!\setminus\!S)|$ given by the number of weighted transversals of the pruned configuration graph $\Gamma(\Ical)\!\setminus\!S$, as $S$ varies over $(r+1)$-element subsets of $[n]$. The smallest upper bound achieved this way is in boldface.

\subsection*{Notation} We work over an algebraically closed field of characteristic zero, denoted $\kfield$. Given a positive integer $n$ we write $[n] \colonequals \{ 1, \ldots, n\}$.

\subsection*{Acknowledgements} We have benefited from helpful conversations with Patricio~Gallardo, Marvin~Anas~Hahn, and Luca~Schaffler. We thank them all for their generosity. The key discoveries of this paper arose during a research visit at the Institute for Advanced Study, which we thank for stellar working conditions.

\subsection*{Funding} 
The first author received support from the Engineering and Physical Sciences Research Council (grant number EP/X001229/1).

\section{Setup}

\subsection{Projective configuration spaces}\label{sec:ConfigurationSpacesSetup}

Fix $r\in\Z$ with $r\geqslant2$ and $n \in \Z$ with $n\geqslant r+1$. Consider
\[ X(r,[n])\]
the moduli space of linearly general configurations of $n$ labelled points in $\P^{r-1}$ up to the action of $\PGL_r$. Being linearly general means that for $l \in \{0,1,\ldots,r-2\}$ there is no linear subspace of dimension $l$ which contains $l+2$ points. The moduli space is a free quotient
\[ X(r,[n]) \colonequals U(r,[n])/\PGL_r \]
where $U(r,[n]) \subseteq (\P^{r-1})^n$ is the open subset parametrising linearly general configurations, and $\PGL_r$ acts diagonally. Note that $\dim(X(r,[n]))=(r-1)(n-r-1)$. Note also that $X(2,n)$ is the moduli space $\M_{0,n}$ of $n$-marked rational curves.

More generally given a finite set $I$ we consider the space $X(r,I)$ of configurations of points labelled by $I$. Projective configuration spaces are equipped with \textbf{forgetful maps}
\[ \fgt_I:X(r,[n])\to X(r,I) \]
for $I\subseteq[n]$ with $|I| \geqslant r+1$.

\subsection{Projective configuration counts} \label{sec: projective configuration counts} Set $k \colonequals n-r-1$. Choose subsets
\[ I_1,\ldots,I_k \in \binom{[n]}{r+2} \]
and write $\Ical = (I_1,\ldots,I_k)$. Consider the map 
$$\fgt_\Ical=\prod_{j=1}^k\fgt_{I_j}:X(r,[n])\to\prod_{j=1}^k X(r,I_j)$$ 
whose $j$th coordinate map is $\fgt_{I_j}$. This is a map of irreducible $(r-1)(n-r-1)$-dimensional varieties, hence has a well-defined degree
\[ d_{r,n}(\mathcal{I}) \colonequals \deg \fgt_\Ical \in \N \]
which we refer to as a \textbf{projective configuration count}. It counts the number of configurations of $n$ points which for each $j \in [k]$ restrict to a fixed configuration of the $r+2$ points in $I_j$.

\subsection{Forgetting and projecting} \label{sec: forget and project} Besides the forgetful maps there are also the \textbf{projection maps}
\[\pr_i:X(r,[n])\to X(r-1,[n]\setminus\{i\})\] 
for $i \in [n]$, which project the points of $[n]\setminus\{i\}$ away from the point $i$, giving a configuration in $\P^{r-2}$. The fact that the original configuration is in linearly general position implies that the same holds after projecting from the point $i$.

\begin{proposition} \label{prop: forgetful and projection are smooth}
    The maps $\fgt_I$ and $\pr_i$ are smooth.
\end{proposition}

\begin{proof} The action $\PGL_r \curvearrowright \P^{r-1}$ is simply $(r+1)$-transitive for points in linearly general position. Fixing a subset $S \subseteq I$ of size $r+1$, there is thus a unique element of $\PGL_r$ which sends the points indexed by $S$ to the standard points $[1,\ldots,0],\ldots,[0,\ldots,1],[1,\ldots,1]$. This rigidification produces an open embedding
\[ X(r,I) \hookrightarrow (\P^{r-1})^{I \setminus S} \]
consisting of the moduli for the remaining points. The same holds for $X(r,[n])$, and there is a commuting diagram
\[
\begin{tikzcd}
    X(r,[n]) \ar[r,hook] \ar[d,"\fgt_I"] & (\P^{r-1})^{[n] \setminus S} \ar[d,"f_I"] \\
    X(r,I) \ar[r,hook] & (\P^{r-1})^{I \setminus S}
\end{tikzcd}
\]
with $f_I$ the projection onto the appropriate factors. A direct examination shows that $X(r,[n])$ is an open subset of the fibre product
\[
\begin{tikzcd}
    X(r,[n]) \ar[r,hook] \ar[rd,"\fgt_I" {yshift=-0.5cm, xshift=-0.4cm}] & Y_I \ar[r,hook] \ar[d] \ar[rd,phantom,"\square" right] & (\P^{r-1})^{[n] \setminus S} \ar[d,"f_I"] \\
    & X(r,I) \ar[r,hook] & (\P^{r-1})^{I \setminus S}
\end{tikzcd}
\]
and since $f_I$ is smooth it follows that $\fgt_I$ is smooth.

For $\pr_I$ we first choose a subset $S \subseteq [n] \setminus \{ i \}$ of size $r$ with which to rigidify, giving an open embedding:
\[ X(r-1,[n] \setminus \{ i\}) \hookrightarrow (\P^{r-2})^{([n]\setminus\{i\})\setminus S}. \]
We then define $S^\prime \colonequals S \sqcup \{i\} \subseteq [n]$ and get an open embedding $X(r,[n]) \hookrightarrow (\P^{r-1})^{[n] \setminus S^\prime}$ which (since $i \in S^\prime$) factors through the open subset
\[ X(r,[n]) \hookrightarrow (\P^{r-1} \setminus \{x_i\})^{[n] \setminus S^\prime} \]
where $x_i \in \P^{r-1}$ is the fixed location to which the $i$th point is sent under the rigidification. Note that $[n] \setminus S^\prime = ([n] \setminus \{i\}) \setminus S$ and we have a commuting diagram
\[
\begin{tikzcd}
    X(r,[n]) \ar[r,hook] \ar[d,"\pr_i"] & (\P^{r-1} \setminus \{x_i\})^{[n] \setminus S^\prime} \ar[d,"p_i"] \\
    X(r-1,[n] \setminus \{ i\}) \ar[r,hook] & (\P^{r-2})^{([n]\setminus\{i\})\setminus S}
\end{tikzcd}
\]
where $p_i$ is the projection $\P^{r-1} \setminus \{ x_i \} \to \P^{r-2}$ in each factor. As in the previous case, smoothness of $p_i$ implies smoothness of $\pr_i$.
\end{proof}

\subsection{Grassmannians and the Gelfand--MacPherson correspondence} \label{sec: Grassmannians} Consider the open subset
\[ \Gr^\circ(r,\Bbbk^{[n]}) \subseteq \Gr(r,\Bbbk^{[n]}) \]
parametrising subspaces with all Pl\"ucker coordinates nonvanishing. The Gelfand--MacPherson correspondence \cite{GelfandMacPherson} identifies $X(r,[n])$ with a quotient thereof:
\[ X(r,[n]) \cong \Gr^\circ(r,\Bbbk^{[n]})/\Gm^{[n]}. \]
The stabiliser at every point is the diagonal $\Gm$. Under this identification, the map $\pr_{i}:X(r,[n])\to X(r-1,[n] \setminus \{ i\})$ lifts to the map 
\[ \pr^{\circ}_{i}:\Gr^{\circ}(r,\Bbbk^{[n]})\to \Gr^{\circ}(r-1,\Bbbk^{[n]\setminus \{i\}}) \]
which sends a subspace $V$ to $V\cap\mathbb{V}(x_{i})$, where $(x_1,\ldots,x_{n})$ are the coordinates on $\Bbbk^{[n]}$ (the nonvanishing of all Pl\"ucker coordinates ensures that this intersection has dimension $r-1$). Similarly the map $\fgt_I \colon X(r,[n])\to X(r,I)$ lifts to the map
\[ \fgt^{\circ}_I:\Gr^{\circ}(r,\Bbbk^{[n]})\to\Gr^{\circ}(r,\Bbbk^{I})\]
which sends a subspace $V$ to its image under the coordinate projection $\Bbbk^{[n]} \to \Bbbk^{I}$ (again, nonvanishing of all Pl\"ucker coordinates ensures that this image has dimension $r$).

\subsection{Matchings, the surplus condition, and positivity} 
Fix data $\Ical = (I_1,\ldots,I_k)$ for a projective configuration count as in Section \ref{sec: projective configuration counts}.

\subsubsection{Matchings} The following notions are used in stating the upper bound in Section \ref{sec: upper bound}.

\begin{definition}[Configuration graph] The associated \textbf{configuration graph}, denoted $\Gamma(\Ical)$, is defined as follows. It is a bipartite graph in which the left vertices are indexed by $[n]$ and the right vertices are indexed by $[k]$, and there is an edge $(i,j)$ connecting $i \in [n]$ and $j \in [k]$ if and only if $i \in I_j$.
\end{definition}

\begin{definition}[Pruned configuration graph] \label{def: pruned configuration graph} Given a subset $S \in \binom{[n]}{r+1}$ the \textbf{pruned configuration graph} $\Gamma(\Ical)\!\setminus\!S$ is the graph obtained from $\Gamma(\Ical)$ by deleting the vertices in $S$ (and all their adjacent edges). Since $n-(r+1) = k$ this is a bipartite graph with the same number of left and right vertices.
\end{definition}

\begin{definition}[Weighted transversal] \label{def: weighted transversal} Fix a bipartite graph $\Gamma$ with the same number of left and right vertices, and an integer $m \geqslant 1$. An \textbf{$m$-weighted transversal} of $\Gamma$ is an assignment of a non-negative integer weight $m_e \in \N$ to every edge $e \in E(\Gamma)$ such that
\[ \sum_{e \ni v} m_e = m \]
for every vertex $v \in V(\Gamma)$. We let $T_m(\Gamma)$ denote the set of $m$-weighted transversals. Note that a $1$-weighted transversal is a perfect matching. \end{definition}

\subsubsection{Surplus condition and positivity}

\begin{definition} 
The \textbf{surplus} $\upsigma(\Ical)$ of the datum $\Ical$ is the quantity 
\[ \upsigma(\Ical):=\min_{\substack{J\subseteq[k]\\J\ne\emptyset}} \left( \abs{\bigcup_{j \in J} I_j} - \abs{J} \right). \] 
Taking $J=[k]$ shows that $\upsigma(\Ical)\leqslant r+1$. We say that $\Ical$ satisfies the \textbf{surplus condition} if $\upsigma(\Ical) = r+1$, equivalently if
\[ \left| \bigcup_{j \in J} I_j \right| \geqslant |J| + r + 1 \]
for all nonempty subsets $J \subseteq [k]$.	
\end{definition}

\begin{proposition} \label{prop: nonzero implies surplus} Suppose that $d_{r,n}(\Ical) \neq 0$. Then $\Ical$ satisfies the surplus condition.
\end{proposition}

\begin{proof} Suppose that the surplus condition fails. This means that there exists a nonempty subset $J \subseteq [k]$ such that
\[ \left| \bigcup_{j \in J} I_j \right| < |J| + r+1.\]
Define $I_J \colonequals \bigcup_{j \in J} I_j$. First consider the case $J=[k]$. In this case we have:
\[ |I_{[k]}| < k+r+1 = n \]
which means that there exists an element of $[n]$ which does not belong to any $I_j$. Taking this element to be $n \in [n]$ without loss of generality, the forgetful map $\fgt_\Ical$ factors through a space of smaller dimension
\[ X(r,[n]) \to X(r,[n\!-\!1]) \to \prod_{j=1}^k X(r,I_j) \]
and so is not dominant; thus $d_{r,n}(\Ical)=0$ as required. For general $J$ we choose a subset $E \subseteq [n]$ of size $|J|+r+1$ extending $I_J$, so that $I_J \subsetneq E$. We then define
\[ \Ical_J \colonequals (I_j \colon j \in J) \]
viewed as a collection of subsets of $E$. Notice that the length of $\Ical_J$ is precisely $|J| = |E|-r-1$. Then consider the commuting diagram
\[
\begin{tikzcd}
X(r,[n]) \ar[r,"\fgt_\Ical"] \ar[d,"\fgt_{E}" left] & \prod_{j \in [k]} X(r,I_j) \ar[d] \\
X(r,E) \ar[r,"\fgt_{\Ical_J}"] & \prod_{j \in J} X(r,I_j).
\end{tikzcd}
\]
The bottom arrow is not dominant by the previous case, with $J$ playing the role of $[k]$ and $E$ playing the role of $[n]$. It follows from Lemma~\ref{lem: dominance square} below that the top arrow is not dominant either, and so $d_{r,n}(\Ical)=0$ as required.
\end{proof}

\begin{lemma} \label{lem: dominance square} Consider a commutative diagram of irreducible varieties:
\[
\begin{tikzcd}
Y^\prime \ar[r,"g"] \ar[d,"q"] & Y \ar[d,"p"] \\
X^\prime \ar[r,"f"] & X.
\end{tikzcd}
\]
Suppose that $g$ is dominant and $p$ is open (which holds in particular if $p$ is flat: see \cite[Exercise~III.9.1]{Hartshorne}). Then $f$ is dominant.	
\end{lemma}

\begin{proof} Since $g$ is dominant there is a dense open subset $U_g$ of $Y$ with $U_g \subseteq g(Y^\prime)$. Since $p$ is open, the image $p(U_g) \subseteq X$ is open. Since it is nonempty, it is dense, and
\[ p(U_g) \subseteq p(g(Y^\prime)) = f(q(Y^\prime)) \subseteq f(X^\prime). \qedhere\]	
\end{proof}

We conjecture the converse to Proposition \ref{prop: nonzero implies surplus}. 
\begin{conjecture} \label{conj: positivity} We have $d_{r,n}(\Ical) \neq 0$ if and only if $\Ical$ satisfies the surplus condition.
	\end{conjecture}

The $r=2$ case is known. See \cite[Theorem~3.3]{JordanKaszanitzky2015} and \cite[Theorem~A]{BrakensiekEurLarsonLi2023} where the surplus condition is referred to as ``$(1,3)$-tightness'' and the ``Cerberus condition'' respectively.
In this setting, characterising when $d_{r,n}(\Ical) \neq 0$ is equivalent to determining the algebraic matroid associated to the embedding
\[ \Mcal_{0,n} \hookrightarrow \prod_{I \in \binom{[n]}{4}} \Mcal_{0,I}. \]
Moreover the basis degrees of this algebraic matroid (see e.g.\ \cite{Rosen2014}) are precisely the (nonzero) cross-ratio degrees. 
For $r \geqslant 3$, the embedding 
\[ X(r,[n]) \hookrightarrow \prod_{I \in \binom{[n]}{r+2}} X(r,I) \]
has an associated \emph{polymatroid} on ground set $[n]$, whose rank function measures the dimensions of the images of all projections. Conjecture~\ref{conj: positivity} would precisely determine the bases of this polymatroid, with the basis degrees being the (nonzero) projective configuration counts.

Proposition \ref{prop: nonzero implies surplus} alternatively follows by combining Theorem \ref{thm: upper bound}, Lemma \ref{lem: no transversals}, and the following result from matching theory, closely related to Hall's marriage theorem (Theorem \ref{thm: Hall}).

\begin{theorem}[{see \cite[Section~1.3]{LPMatching}}] \label{thm: surplus iff matching} The surplus condition holds if and only if $\Gamma(\Ical)\!\setminus\!S$ admits a perfect matching for all $S \in \binom{[n]}{r+1}$.
\end{theorem}

\begin{remark} Theorem \ref{thm: surplus iff matching} holds in the general context where $\Ical$ is an arbitrary set system and $r+1$ is an arbitrary natural number. Here the notion of perfect matching is replaced by that of a transversal. This explains our terminology in Definition \ref{def: weighted transversal}. \end{remark}

\section{Upper bound} \label{sec: upper bound}

\noindent We establish a graph-theoretic upper bound for the projective configuration counts (Theorem \ref{thm: upper bound}). This generalises the $r=2$ result \cite[Theorem~1.1]{Silversmith2022}. It is interesting to note that the naive extension of the graph theoretic count in the $r=2$ result fails for $r \geqslant 3$, and genuinely new combinatorics is required; see Example \ref{example naive upper bound fails} and Remark \ref{rem: need weighted transversals}.

Fix data $\Ical=(I_1,\ldots,I_k)$ for a projective configuration count, and $S \in \binom{[n]}{r+1}$. Recall the pruned configuration graph $\Gamma(\Ical)\!\setminus\!S$ (Definition \ref{def: pruned configuration graph}) and the notion of weighted transversal (Definition \ref{def: weighted transversal}).

\begin{theorem}[Theorem \ref{thm: upper bound introduction}] \label{thm: upper bound} The projective configuration count is bounded above by the number of $(r-1)$-weighted transversals of the pruned configuration graph $\Gamma(\Ical)\!\setminus\!S$:
\[ d_{r,n}(\mathcal{I}) \leqslant | T_{r-1}(\Gamma(\Ical)\!\setminus\!S) |. \]
\end{theorem}

\subsection{Proof outline} Theorem \ref{thm: upper bound} generalises \cite[Theorem~1.1]{Silversmith2022} which was given a different proof in \cite[Theorem~C]{BrakensiekEurLarsonLi2023}. We follow the former argument. The proof proceeds in four steps:
\begin{itemize}
    \item Section \ref{sec: upper bound proof 1}. We equate the projective configuration count with an intersection multiplicity in $(\PGL_r)^k$ (Lemma \ref{lem: projective configuration count equals intersection mult}).
    \item Section \ref{sec: upper bound proof 2}. By passing to the compactification $(\PGL_r)^k \subseteq (\P^R)^k$ (where $R \colonequals r^2-1$), we bound from above the projective configuration count by an intersection product in $(\P^R)^k$ (Lemma \ref{lem: upper bound in terms of product of lambdai}).
    \item Section \ref{sec: upper bound proof 3}. We determine explicitly the cycles in $(\P^R)^k$ we wish to intersect (Propositions~\ref{prop: lambdai class case 1} and \ref{prop: lambdai class case 2}).
    \item Section \ref{sec: upper bound proof 4}. We equate the intersection product of these cycles with the number of weighted transversals (Lemma \ref{lem: product of lambda equals transversals}).
\end{itemize}

\begin{remark} \label{remark: assume each Ji nonempty} For the rest of this section we will assume that the starting data has been chosen such that every $i \in [n]$ belongs to $I_j$ for some $j \in [k]$. Otherwise, the projective configuration count vanishes because $\fgt_\Ical$ factors through the forgetful morphism
\[ X(r,[n]) \to X(r,[n] \setminus \{ i \} ). \]
The upper bound is trivially satisfied in this case.
\end{remark}

\subsection{Intersection problem} \label{sec: upper bound proof 1}
We first modify our input data as follows. We enlarge our marking index set to $[n]_0 \colonequals [n] \sqcup\{0\}$, enlarge our constraint index set to $[k]_0 \colonequals [k] \cup\{0\}$, and enlarge our list of constraints to $\Ical_0 \colonequals (I_0,I_1,\ldots,I_k)$ where $I_0 \colonequals S \sqcup\{0\}$. Note that
\[ d_{r,n}(\Ical) = d_{r,[n]_0}(\Ical_0) \]
since the new marking $0$ is uniquely determined by the constraint $I_0$. We fix generic points
\begin{equation} \label{eqn: fixed points} \pt^i_j \in \PP^{r-1} \end{equation}
for $j \in [k]_0$ and $i \in I_j$. The projective configuration count $d_{r,[n]_0}(\Ical_0)$ enumerates $p_0,p_1,\ldots,p_n \in \PP^{r-1}$ up to the diagonal action of $\PGL_r$, such that for each $j \in [k]_0$ there exists a (necessarily unique) element $g_j \in \PGL_r$ with
\begin{equation} \label{eqn: using PGLr to hit element} g_j^{-1}(p_i) = \pt^i_j \end{equation}
for all $i \in I_j$.\footnote{The inverse sign convention will declutter later arguments.}  A choice of $p_i$ uniquely determines a choice of $g_j$, and \eqref{eqn: using PGLr to hit element} implies
\begin{equation} \label{eqn: relation between gj} g_{j_1}(\pt^i_{j_1}) = g_{j_2}(\pt^i_{j_2}) \end{equation}
for all $i \in I_{j_1} \cap I_{j_2}$. Conversely, given $g_j$ satisfying \eqref{eqn: relation between gj} we can define $p_i \colonequals g_j(\pt^i_j)$ for any $j$ such that $i \in I_j$; this is well-defined by \eqref{eqn: relation between gj}. Therefore, rather than enumerating the points $p_0,p_1,\ldots,p_n \in \PP^{r-1}$ it is equivalent to enumerate the elements $g_0,g_1,\ldots,g_k \in \PGL_r$ satisfying \eqref{eqn: relation between gj}. Moreover we may rigidify the diagonal $\PGL_r$ action by setting $p_i = \pt^i_j$ for $i \in I_0$, which is equivalent to setting $g_0=\Id$.

We now construct the locus of such $g_j$, and prove that it is smooth of the expected dimension. Consider embeddings
\[ \P^{r-1} \hookrightarrow (\P^{r-1})^{[k]_0} \]
such that each composite $\P^{r-1} \hookrightarrow (\P^{r-1})^{[k]_0} \to \P^{r-1}$ is an automorphism. We consider the moduli space of such embeddings, up to automorphisms of the domain. This is isomorphic to the quotient
\[ \PGL_r \backslash (\PGL_r)^{[k]_0}. \]
We use our new index $0 \in [k]_0$ to rigidify the action, producing an isomorphism
\[ \PGL_r \backslash (\PGL_r)^{[k]_0} \cong (\PGL_r)^k. \]
This space carries a universal family, which under the above isomorphism takes the form
\[
\begin{tikzcd}
    (\PGL_r)^k \times \P^{r-1} \ar[r,"\upsigma"] \ar[d,"\uppi"] & (\P^{r-1})^{[k]_0} \\
    (\PGL_r)^k
\end{tikzcd}
\]
where $\upsigma \colon (g_1,\ldots,g_k,z) \mapsto (z,g_1^{-1}z,\ldots,g_k^{-1} z)$.

For $i \in [n]_0$ we define a subvariety $\tilde{\Lambda}_i(\pt^i)$ of $(\PGL_r)^k \times \PP^{r-1}$ as follows. First consider the set indexing tuples $I_j$ which contain $i$,
\[ J_i \colonequals \{ j \in [k]_0 \colon i \in I_j \}, \]
noting that $i \in I_j$ if and only if $j \in J_i$. Consider the projection
\[ \uppi_{J_i} \colon (\P^{r-1})^{[k]_0} \to (\P^{r-1})^{J_i}. \]
Recall the points $\pt^i_j$ chosen in \eqref{eqn: fixed points} and arrange these according to the associated marking by setting:
\[ \pt^i \colonequals (\pt^i_j)_{j \in J_i} \in (\P^{r-1})^{J_i}.\]
We then define
\begin{align*} \tilde{\Lambda}_i(\pt^i) & \colonequals \upsigma^{-1} \uppi_{J_i}^{-1}(\pt^i) \\
& = \{ (g_1,\ldots,g_k,z) \in (\PGL_r)^k \times \P^{r-1} : g_j^{-1} z = \pt^i_j \text{ for all } j \in J_i \}
\end{align*}
where by our choice of rigidification we have
\begin{equation} \label{eqn: g0 is Id} g_0 = \Id \in \PGL_r.\end{equation}
As a subscheme of $(\PGL_r)^k \times \P^{r-1}$ this is defined by the equations
\begin{align*}
    g_{j_1} \pt^i_{j_1} & = g_{j_2} \pt^i_{j_2} \qquad \text{for all $j_1,j_2 \in J_i$,} \\
    z & = g_j \pt^i_j \qquad \ \ \, \text{for all $j \in J_i$},
\end{align*}
and $z$ plays the role of the marking $p_i$. We now consider the image of $\tilde{\Lambda}_i(\pt^i)$ under $\uppi$. On the level of closed points this is given by
\[
\Lambda_i(\pt^i) = \{(g_1,\ldots,g_k)\in(\PGL_r)^k : \exists z \in \P^{r-1} \text{ such that } g_j^{-1} z = \pt^i_j \text{ for all } j \in J_i \}
\]
while as a subscheme of $(\PGL_r)^k$ it is defined by the equations:
\begin{equation} \label{eqn: description of Lambdai not involving z} \Lambda_i(\pt^i) = \left\{ (g_1,\ldots,g_k)\in(\PGL_r)^k : g_{j_1} \pt^i_{j_1} = g_{j_2} \pt^i_{j_2} \text{ for all $j_1,j_2 \in J_i$} \right\}.\end{equation}
Note that these are equivalent to \eqref{eqn: relation between gj}. Comparing these to the equations cutting out $\tilde{\Lambda}_i(\pt^i)$, and using the assumption that $J_i$ is nonempty (see Remark \ref{remark: assume each Ji nonempty}), we see that once $(g_1,\ldots,g_k)$ are fixed the point $z$ is determined, and more specifically determined by a system of linear equations in the homogeneous coordinates on $\P^{r-1}$. Therefore $\uppi$ restricts to an isomorphism:
\[ \tilde{\Lambda}_i(\pt^i) \cong \Lambda_i(\pt^i). \]

\begin{lemma} \label{lem: codim of Lambdai} $\Lambda_i(\pt^i)$ is a smooth variety of codimension $(r-1)(|J_i|-1)$ in $(\PGL_r)^k$. \end{lemma}

\begin{proof} The action map $\upsigma$ is smooth and hence so is the composite $\uppi_{J_i} \circ \upsigma$. It follows that $\tilde{\Lambda}_i(\pt^i)$ is smooth with
\[ \codim \tilde{\Lambda}_i(\pt^i) = \codim (\pt^i) = (r-1) |J_i|. \]
We saw above that the projection $\uppi \colon (\PGL_r)^k \times \P^{r-1} \to (\PGL_r)^k$ is an isomorphism when restricted to $\tilde{\Lambda}_i(\pt^i)$. It follows that $\Lambda_i(\pt^i)$ is smooth, and
\[ \codim \Lambda_i(\pt^i) = \codim \tilde{\Lambda}_i(\pt^i) - \operatorname{reldim}(\uppi) = (r-1)(|J_i|-1). \qedhere\]
\end{proof}

\begin{lemma} \label{lem: projective configuration count equals intersection mult} For generic choices of $\pt^i_j$ we have
\begin{equation}\label{eq: projective configuration count equals intersection mult}    
d_{r,[n]_0}(\Ical_0) = \left| \bigcap_{i \in [n]_0} \Lambda_i(\pt^i) \right|.
\end{equation}
\end{lemma}

\begin{proof}
The intersection of the $\Lambda_i(\pt^i)$ is the set of embeddings that satisfy all the incidence constraints simultaneously, and is thus equal to a fibre of $\fgt_\Ical$. By \cite[Theorem~8.2(b)]{Fulton1998} and the smoothness of $\Lambda_i(\pt^i)$ it suffices to check that the intersection is transverse, i.e.\ that it consists of finitely many reduced points.

Consider the right action of $(\PGL_r)^k$ on itself. By Kleiman's theorem \cite{Kleiman1974} (see also \cite[Section~B.9.2]{Fulton1998}), for generic elements $h_i \in (\PGL_r)^k$ the intersection
\[ \bigcap_{i \in [n]_0} \Lambda_i(\pt_i) h_i \]
is transverse (smooth of the expected dimension). Now for $i \in [n]_0$ there is a projection
\[(\PGL_r)^k \to (\PGL_r)^{J_i}\]
which produces a left action $(\PGL_r)^k \curvearrowright (\P^{r-1})^{J_i}$. Then it follows directly from \eqref{eqn: description of Lambdai not involving z} that
\[ \Lambda_i(\pt^i) h_i = \Lambda_i(h_i^{-1} \pt^i).\]
Replacing each $\pt^i$ by $h_i^{-1} \pt^i$ in \eqref{eq: projective configuration count equals intersection mult} completes the proof.
\end{proof}
From now on we fix generic $\pt^i_j$ as in the preceding lemma. We will write $\Lambda_i$ instead of $\Lambda_i(\pt^i)$ when no confusion can arise.

\subsection{Compactification and upper bound} \label{sec: upper bound proof 2}
Each $\PGL_r$ admits a compactification\footnote{We embed $\PGL_r$ into $\P^R$ by the matrix entries. There is another natural embedding $\PGL_r\hookrightarrow\P^R$ given by first inverting the matrix.}
\[ \PGL_r \hookrightarrow \P^R \]
where $R \colonequals r^2-1$. We will use intersection theory on this compactification to bound the number of intersection points. To this end, we consider the Chow class of the closure of $\Lambda_i$ in $(\P^R)^k$:
\[ [\overline{\Lambda}_i] \in A^\star((\P^R)^k). \]

\begin{lemma} \label{lem: upper bound in terms of product of lambdai} The projective configuration count $d_{r,n}(\Ical)$ is bounded from above by the product of the $[\overline{\Lambda}_i]$:
\[ d_{r,n}(\Ical) \leqslant \prod_{i \in [n]_0} [\overline{\Lambda}_i]. \]
\end{lemma}
\begin{proof} This argument is similar to \cite[proof~of~Theorem~C]{BrakensiekEurLarsonLi2023}. By Lemma \ref{lem: projective configuration count equals intersection mult} the intersection $\bigcap_{i=0}^n \Lambda_i$ is zero-dimensional. It follows that when we pass to the compactification we have a decomposition
\[
\bigcap_{i=0}^n \overline{\Lambda}_i = \bigcap_{i=0}^n \Lambda_i \sqcup Z,
\]
where $Z$ is the piece of the intersection supported on the boundary. By \cite[Definition~6.1.2]{Fulton1998} this admits distinguished varieties
\[ Z_1,\ldots,Z_m \subseteq Z, \qquad \bigcup_{i=1}^m Z_i = Z, \]
obtained by projecting the irreducible components of the normal cone of $Z$ in the external product $\Pi_{i=0}^n \overline{\Lambda}_i$. Each distinguished variety gives a canonical contribution to the intersection product, and we will show that this contribution is non-negative. This implies the result since by Lemma \ref{lem: projective configuration count equals intersection mult} the first piece $\bigcap_{i=0}^n \Lambda_i$ of the intersection contributes precisely $d_{r,n}(\Ical)$.

Non-negativity follows from the non-negativity of the tangent bundle of $(\P^R)^k$ (see \cite[Chapter~12]{Fulton1998}). Precisely, consider the following diagram:
\[
\begin{tikzcd}
\bigcap_{i=0}^n \overline{\Lambda}_i \ar[r] \ar[d] \ar[rd,phantom,"\square"] & \prod_{i=0}^n \overline{\Lambda}_i \ar[d] \\
(\PP^R)^k \ar[r,hook,"\Delta"] & \prod_{i=0}^n (\P^R)^k.
\end{tikzcd}
\]
The tangent bundle of $(\P^R)^k$ is globally generated, since by the Euler sequence it arises as a quotient of the direct sum
\[ \OO(1,0,\ldots,0)^{\oplus(R+1)} \oplus \cdots \oplus \OO(0,\ldots,0,1)^{\oplus(R+1)} \]
which is itself globally generated. Its restriction to each distinguished variety $Z_i$ is thus also globally generated, and \cite[Corollary~12.2(a)]{Fulton1998} implies that the contribution of $Z_i$ to the intersection product is non-negative.
\end{proof}

\subsection{Determining the cycles} \label{sec: upper bound proof 3} We will show that the product $\Pi_{i \in [n]_0} [\overline{\Lambda}_i]$ is equal to the number of $(r-1)$-weighted transversals of $\Gamma(\Ical)\!\setminus\!S$. The first step is to express each $[\overline{\Lambda}_i]$ in terms of the hyperplane classes $H_1,\ldots,H_k$ on $(\P^R)^k$. There are two distinct cases:
\begin{enumerate}
    \item Proposition \ref{prop: lambdai class case 1}: $i \in I_0$ (equivalently, $0 \in J_i$).
    \item Proposition \ref{prop: lambdai class case 2}: $i \not\in I_0$ (equivalently, $0 \not\in J_i$).
\end{enumerate}
We deal with these in turn.

\begin{proposition} \label{prop: lambdai class case 1} If $i \in I_0$ then
\[ [\overline{\Lambda}_i] = \prod_{j\in J_i\setminus\{0\}} H_j^{r-1}. \]
\end{proposition}

\begin{proof}
By \eqref{eqn: g0 is Id} we have $g_0 = \Id$, so \eqref{eqn: description of Lambdai not involving z} gives 
\[ \Lambda_i = \{(g_1,\ldots,g_k)\in(\PGL_r)^k : \pt_0^i = g_j \pt^i_j \text{ for all } j \in J_i \setminus \{0\} \}. \]
This is a codimension~$(r-1)$ linear condition on the entries of each $g_j$. Such a description remains true when we pass to the compactification, producing the desired formula.
\end{proof}

\begin{proposition} \label{prop: lambdai class case 2} If $i \not\in I_0$ then
\[ [\overline{\Lambda}_i] = \sum_{\substack{(a_{ij})_{j\in J_i}\\0\leqslant a_{ij}\leqslant r-1\\\sum_{j\in J_i}a_{ij}=(r-1)(\abs{J_i}-1)}}\prod_{j\in J_i}H_j^{a_{ij}}. \]
\end{proposition}

\begin{proof} By \eqref{eqn: description of Lambdai not involving z} we have 
\[\Lambda_i=\{(g_1,\ldots,g_k)\in(\PGL_r)^k:g_{j_1}\pt_{j_1}^i=g_{j_2}\pt_{j_2}^i \text{ for all } j_1,j_2\in J_i\}.\]
Given $(g_1,\ldots,g_k) \in (\PGL_r)^k$ let $M(g_\bullet)$ be the $r\times\abs{J_i}$ matrix whose $j$\/th column (for $j\in J_i$) is an arbitrary affine representative of $g_j\pt_j^i.$ Then the above becomes 
\[\Lambda_i=\{(g_1,\ldots,g_k)\in(\PGL_r)^k:\rk(M(g_\bullet))=1\}.\]
Note that the $(\ell,j)$ entry of $M(g_\bullet)$ is a linear form in the entries of the matrix $g_j$. Passing to the compactification, these become the homogeneous coordinates of the $j$\/th factor of $(\P^R)^k$ and $M(g_\bullet)$ extends to a map of vector bundles
\[ \bigoplus_{j\in J_i} \O_{(\P^R)^k}(-H_j) \to\O^{\oplus r} \]
over $(\PP^R)^k$ whose rank-$1$ degeneracy locus is $\overline{\Lambda}_i$. Since this is of the expected dimension (Lemma \ref{lem: codim of Lambdai}) and the ambient space is Cohen--Macaulay, the Thom--Porteous formula \cite[Theorem~14.4(c)]{Fulton1998} gives
\[ [\overline{\Lambda}_i] = \det(C), \]
where $C$ is the $(\abs{J_i}-1)\times(\abs{J_i}-1)$ matrix whose $(a,b)$ entry is the following Chern class on $(\PP^R)^k$:
\[c_{r-1+b-a}\left(\O^{\oplus r}-\bigoplus_{j\in J_i}\O(-H_j)\right).\]
We calculate this as
\begin{align*}
    c\left(\O^{\oplus r}-\bigoplus_{j\in J_i}\O(-H_j)\right)&= c(\O^{\oplus r}) \cdot c\left(\bigoplus_{j\in J_i}\O(-H_j)\right)^{-1} \\[0.2cm]
    &=1 \cdot \prod_{j\in J_i}(1-H_j)^{-1}\\[0.2cm]
    &=\sum_{\ell\ge0}h_\ell((H_j)_{j\in J_i}),
\end{align*}
where $h_\ell$ is the complete homogeneous symmetric function. In particular, 
\[c_{r-1+b-a} (\O^{\oplus r}-\oplus_{j\in J_i}\O(-H_j) )=h_{r-1+b-a}((H_j)_{j\in J_i}).\]
The first Jacobi--Trudi identity then applies, giving 
\[ [\overline{\Lambda}_i] = \det(C) = s_{(r-1)^{|J_i|-1}}((H_j)_{j\in J_i}),\]
where the right-hand side is the Schur polynomial associated to the partition
\[ (r-1)^{|J_i|-1} \colonequals (\underbrace{r-1,\ldots,r-1}_{\abs{J_i}-1\text{ times}}).\]
By \cite[Exercise 7.41]{StanleyEC2} we have 
\[  s_{(r-1)^{|J_i|-1}}((H_j)_{j\in J_i}) = h_{r-1}((H_j^{-1})_{j\in J_i})\cdot\prod_{j\in J_i}H_j^{r-1}, \]
which expands out to give the desired formula for $[\overline{\Lambda}_i]$.
\end{proof}

\subsection{Weighted transversals} \label{sec: upper bound proof 4}

Having determined the $[\overline{\Lambda}_i]$ in the previous section, we now equate their product with the number of weighted transversals:

\begin{lemma} \label{lem: product of lambda equals transversals} $\prod_{i=0}^n [\overline{\Lambda}_i] = |T_{r-1}(\Gamma(\Ical)\!\setminus\!S)|$.   
\end{lemma}

\begin{proof} The intersection number is equal to the coefficient of $\prod_{j=1}^k H_j^R$ in the product of the $[\overline{\Lambda}_i]$. The coefficients appearing in the formulae in Propositions~\ref{prop: lambdai class case 1} and \ref{prop: lambdai class case 2} are either $0$ or $1$. The intersection number is thus equal to the number of ways to select a monomial $z_i$ from the formula for $[\overline{\Lambda}_i]$ such that $\Pi_{i=0}^n z_i = \Pi_{j=1}^k H_j^R$. 

For $i \in I_0$ (Proposition \ref{prop: lambdai class case 1}) there is no choice to make, and we must have:
\begin{equation} \label{eqn: transversals proof first monomial} z_i = \prod_{j\in J_i\setminus\{0\}} H_j^{r-1}. \end{equation}
For $i \not\in I_0$ (Proposition \ref{prop: lambdai class case 2}) we can choose any monomial of the form
\begin{equation} \label{eqn: transversals proof second monomial} z_i = \prod_{j\in J_i}H_j^{a_{ij}} \end{equation}
for some $(a_{ij})_{j\in J_i}$ such that:
\begin{equation} \label{eqn: transversals proof first sum}  \sum_{j \in J_i} a_{ij} = (r-1)(|J_i|-1). \end{equation}
The condition $\prod_{i=0}^n z_i = \prod_{j=1}^k H_j^R$ implies that for each $j\in[k]$:
\[(r-1)\cdot| I_j\cap I_0| + \sum_{i\in I_j\setminus I_0} a_{ij} = R = r^2-1\]
where the two summands come from \eqref{eqn: transversals proof first monomial} and \eqref{eqn: transversals proof second monomial}. Subtracting each side from $\sum_{i\in I_j}(r-1) = (r-1)(r+2)$, we have
\begin{equation} \label{eqn: transversals proof second sum}
\sum_{i\in I_j\setminus I_0} (r-1-a_{ij}) = (r-1)(r+2) - (r^2-1) = r-1.
\end{equation}
Now observe that we have defined $a_{ij}$ exactly when $(i,j)$ is an edge of $\Gamma(\Ical)\!\setminus\!S$. Indeed: $a_{ij}$ is defined for $i \in [n]_0 \setminus I_0 = [n] \setminus S$ and $j \in J_i$. The condition $i \not\in I_0$ corresponds to removing the vertices in $S$ to pass from $\Gamma(\Ical)$ to $\Gamma(\Ical)\!\setminus\!S$, while the condition $j \in J_i$ is equivalent to $i \in I_j$, and ensures that $(i,j)$ is an edge of $\Gamma(\Ical)\!\setminus\!S$.

For such a choice of $(a_{ij})$, we define a weighting on the edges of $\Gamma(\Ical)\!\setminus\!S$ by giving the edge $(i,j)$ the weight $r-1-a_{i,j} \geqslant 0$. It follows from \eqref{eqn: transversals proof first sum} (respectively \eqref{eqn: transversals proof second sum}) that the sum of the weights on edges incident to the left vertex $i\in[n]_0\setminus I_0 = [n] \setminus S$ (respectively $j \in [k]$) is $r-1$. That is, this weighting is an $(r-1)$-weighted transversal.

The above assignment of a transversal to a choice of monomials $z_i$ is visibly bijective, which completes the proof.
\end{proof}

\begin{proof}[Proof of Theorem \ref{thm: upper bound}]
By Lemma \ref{lem: upper bound in terms of product of lambdai} we have
\[ d_{r,n}(\Ical) \leqslant \prod_{i=0}^n [\overline{\Lambda}_i] \]
whereas Lemma \ref{lem: product of lambda equals transversals} gives
\[ \prod_{i=0}^n [\overline{\Lambda}_i] = |T_{r-1}(\Gamma(\Ical)\!\setminus\!S)|. \]
Combining these, we obtain the result.
\end{proof}

\section{Dimension reduction}

\noindent The following theorem gives a compatibility between the projective configuration counts of $X(r,n)$ and $X(r\!+\!1,n\!+\!1)$. Note that both cases have the same number of constraints: $k=n-r-1=(n+1)-(r+1)-1$.
\begin{theorem}[Theorem \ref{thm: dimension reduction introduction}] \label{thm: dimension reduction}
    Fix a tuple $\mathcal{I}=(I_1,\ldots,I_k) \in\binom{[n]}{r+2}^k$ and append an additional marking to each constraint:
    \[ \mathcal{I}' \colonequals (I_1\sqcup\{n\!+\!1\},\ldots,I_k\sqcup\{n\!+\!1\}) \in\binom{[n\!+\!1]}{(r\!+\!1)+2}^k.\]
    Then $d_{r+1,n+1}(\mathcal{I}') = d_{r,n}(\mathcal{I})$.
\end{theorem}

The proof, which involves detailed linear algebra and Pl\"ucker calculus, proceeds via a sequence of reductions:
\[ \text{Theorem \ref{thm: dimension reduction}} \xLeftarrow{\text{\tiny{Step \ref{step: describing the preimage}}}} \text{Proposition \ref{prop: Gm n plus one orbit}} \xLeftarrow{\text{\tiny{Step \ref{step: Gm orbits}}}} \text{Proposition \ref{prop: Gm orbit}} \xLeftarrow{\text{\tiny{Step~\ref{step: intersecting conditions}}}} \text{Proposition \ref{prop: inclusion of intersection of Hjbar}} \xLeftarrow{\text{\tiny{Step \ref{step: swapping Plucker}}}} \text{Proposition \ref{prop: A has nonvanishing Plucker}} \]
The proof is broken is broken down first into sections (\ref{sec: dimension reduction proof setup}--\ref{sec: dimension reduction proof combining conditions}), and then into steps (\ref{step: fibre counting}--\ref{step: I vanishing}). Throughout, we indicate the first appearance of the most important symbols by enclosing them in a box.

\subsection{Setup} \label{sec: dimension reduction proof setup} We begin by reducing Theorem \ref{thm: dimension reduction} to a lifting problem involving coordinate-general linear subspaces.

\proofstep{Fibre counting} \label{step: fibre counting} Using the projection and forgetful maps from Section \ref{sec: forget and project} we have the commutative diagram
    \begin{equation}\label{eq:PrFgtSquare}
    \begin{tikzcd}
        X(r\!+\!1,[n\!+\!1])\arrow[rr,"\fgt_{\Ical'}"]\arrow[d,"\pr_{n+1}"]&&\prod_{j=1}^k X(r\!+\!1,I_j^\prime)\arrow[d,"\prod_{j=1}^k\pr_{n+1}"]\\
        X(r,[n])\arrow[rr,"\fgt_\Ical"]&&\prod_{j=1}^k X(r,I_j)
    \end{tikzcd}
    \end{equation}
    where $I_j^\prime \colonequals I_j \sqcup \{ n\!+\!1\}$. Fix a generic point $(x_1^\prime,\ldots,x_k^\prime) \in \prod_{j=1}^k X(r\!+\!1,I_j^\prime)$ and let $x_j \colonequals \pr_{n+1}(x_j^\prime) \in X(r,I_j)$ for $j \in [k]$. We then have
    \begin{align*} d_{r+1,n+1}(\Ical')& = \len(\fgt_{\Ical'}^{-1}(x_1^\prime,\ldots,x_k^\prime)), \\
    d_{r,n}(\Ical) & =\len(\fgt_{\Ical}^{-1}(x_1,\ldots,x_k)).
    \end{align*} 
    We will show that these two preimages are isomorphic. The left vertical morphism restricts to a morphism
    \[ \pr_{n+1} \colon \fgt_{\Ical^\prime}^{-1}(x_1^\prime,\ldots,x_k^\prime) \to \fgt_\Ical^{-1}(x_1,\ldots,x_k) \]
    and we will construct an inverse. This is equivalent to showing that under the morphism to the fibre product
    \begin{equation}\label{eq:MapToFiberProduct}
    X(r\!+\!1,[n\!+\!1])\to X(r,[n])\times_{\prod_{j=1}^k X(r,I_j)}\prod_{j=1}^k X(r\!+\!1,I_j^\prime)
    \end{equation}
     a generic point in the target has a unique preimage. Note that by Proposition \ref{prop: forgetful and projection are smooth} the vertical morphisms in \eqref{eq:PrFgtSquare} are smooth, so both the source and target of \eqref{eq:MapToFiberProduct} are $r(n-r-1)$-dimensional smooth varieties. 
    
    By Lemma \ref{lem: dominance square}, if $\fgt_\Ical$ is not dominant then neither is $\fgt_{\Ical^\prime}$. In this case we have $d_{r,n}(\Ical)=d_{r+1,n+1}(\Ical')=0$ and the theorem holds trivially. Thus without loss of generality we assume that
\begin{equation} \label{assumption count nonzero} d_{r,n}(\Ical) \neq 0. \end{equation}
This assumption will be used only in Step~\ref{step: I vanishing}, during the proof of Lemma \ref{lem: general I vanishing matrix is coordinate-general}.

\proofstep{Reformulating via Grassmannians} \label{step: Grassmannians} We use the identification
    \[ X(r,[n])=\Gr^{\circ}(r,\Bbbk^{[n]})/\Gm^{[n]} \] 
    described in Section \ref{sec: Grassmannians}. Recall that $\Gr^\circ(r,\kfield^{[n]})$ parametrises linear subspaces with all Pl\"ucker coordinates nonvanishing. We refer to such subspaces as \textbf{coordinate-general}.
    
Associated to \eqref{eq:PrFgtSquare} is the following diagram of prequotients:
\[
\begin{tikzcd}
    \Gm^{[n+1]} \curvearrowright \Gr^\circ(r\!+\!1, \Bbbk^{[n+1]}) \ar[r,"\fgt^\circ_{\Ical^\prime}"] \ar[d,"\pr^\circ_{n+1}", xshift=1.25cm] & \prod_{j=1}^k \Gr^\circ(r\!+\!1,\Bbbk^{I_j^\prime}) \curvearrowleft \prod_{j=1}^k \Gm^{I_j^\prime} \ar[d,"\prod_{j=1}^k \pr^\circ_{n+1}", xshift=-2.25cm] \\
   \makebox[\widthof{$\Gm^{[n+1]} \curvearrowright \Gr^\circ(r+1,\Bbbk^{[n+1]})$}][r]{$\Gm^{[n]} \curvearrowright \Gr^\circ(r,\Bbbk^{[n]}) \ar[r,"\fgt^\circ_\Ical"]$} \ar[r,"\fgt^\circ_\Ical"] & \prod_{j=1}^k \Gr^\circ(r,\Bbbk^{I_j}) \curvearrowleft \prod_{j=1}^k \Gm^{I_j}\ .
\end{tikzcd}
\]
Recall here that ${\pr}_{n+1}^\circ$ sends a vector subspace to its intersection with the coordinate hyperplane $\mathbb{V}(x_{n+1})$, while ${\fgt}_I^\circ$ sends a vector subspace to its image along the coordinate projection $p_I$.

\proofstep{Fixing the point in the target} \label{step: fixing target point} Fix a generic point in the target of \eqref{eq:MapToFiberProduct} and choose a representative at the level of Grassmannians. This consists of the following data:
\begin{itemize}
\item $\boxed{V} \subseteq \kfield^{[n]}$ with $\dim V = r$ and $V$ coordinate-general.
\item $\boxed{V_j^\prime} \subseteq \kfield^{I_j^\prime}$ with $\dim V_j^\prime=r\!+\!1$ and $V_j^\prime$ coordinate-general, for each $j \in [k]$.
\end{itemize}
We then define
\[ \boxed{V_j} \colonequals \fgt^\circ_{I_j}(V) = p_{I_j}(V) \subseteq \kfield^{I_j} \]
and since we have chosen a point in the fibre product, there exist $\boxed{\ul{t}_j} \in \Gm^{I_j}$ such that
\[ V_j^\prime \cap \mathbb{V}(x_{n+1}) = \ul{t}_j V_j. \]
We refer to the above as \textbf{starting data}.

\proofstep{Describing the preimage} \label{step: describing the preimage} We now describe the preimage under \eqref{eq:MapToFiberProduct} of the point chosen in the previous step. At the level of Grassmannians the preimage corresponds to the locus
\[ \Gcal (\text{A},\text{B}_j) \subseteq \Gr^\circ(r\!+\!1,\kfield^{[n+1]}) \] 
consisting of subspaces $V^\prime \subseteq \kfield^{[n+1]}$ satisfying the following two conditions:
\begin{enumerate}
\item[($\text{A}$)] $V^\prime \cap \mathbb{V}(x_{n+1}) = \ul{t} V$ for some $\ul{t} \in \Gm^{[n]}$.
\item[($\text{B}_j$)] $p_{I_j^\prime}(V^\prime) = \ul{t}_j^\prime V_j^\prime$ for some $\ul{t}_j^\prime \in \Gm^{I_j^\prime}$.
\end{enumerate}
Here $(\text{B}_j)$ must hold for all $j \in [k]$. Diagrammatically, producing an element of $\Gcal(\text{A},\text{B}_j)$ amounts to solving the following problem:
\[
\begin{tikzcd}
\color{red}{V^\prime} \ar[d,|->,"\cap \mathbb{V}(x_{n+1})" left] \ar[r,|->,"p_{I_j^\prime}"] & \color{red}{\ul{t}_j^\prime\,} \color{black}{\cdot V_j^\prime} & V_j^\prime \ar[d,|->,"\cap \mathbb{V}(x_{n+1})"] \\
\color{red}{\ul{t}\,} \color{black}{\cdot V} & & \ul{t}_j V_j \\
V \ar[r,|->,"p_{I_j}"] & V_j & &
\end{tikzcd}
\]
where the black data is the starting data fixed in Step~\ref{step: fixing target point}, and the red data is what we need to find. We note that $\Gcal(\text{A},\text{B}_j)$ is a union of orbits for the action
\[ \Gm^{[n+1]} \curvearrowright \Gr^\circ(r\!+\!1,\kfield^{[n+1]}) \]
and Theorem \ref{thm: dimension reduction} reduces to the following:
\begin{proposition} \label{prop: Gm n plus one orbit} The locus $\Gcal(\text{A},\text{B}_j)$ in fact consists of a single $\Gm^{[n+1]}$-orbit.
\end{proposition}

\subsection{Describing the conditions} To establish Proposition \ref{prop: Gm n plus one orbit}, we first reduce it in Step~\ref{step: Gm orbits} to a statement involving $\Gm$-orbits (Proposition \ref{prop: Gm orbit}). We then unravel separately the conditions (A) and ($\text{B}_j$), in Steps~\ref{step: open locus}~and~\ref{step: locally closed}.

\proofstep{Reduction to $\Gm$-orbits} \label{step: Gm orbits} Given $V^\prime \in \Gcal(\text{A},\text{B}_j)$ we have
\[ V^\prime \cap \mathbb{V}(x_{n+1}) = \ul{t} \cdot V \]
for some $\ul{t} \in \Gm^{[n]}$. Moreover, because $V$ is coordinate-general, $\ul{t}$ is unique up to diagonal scaling \cite[Proposition~1.1.2]{GelfandMacPherson}. We then have
\[ ((\ul{t}^{-1},1) \cdot V^\prime) \cap \mathbb{V}(x_{n+1}) = V \]
and this identity persists if we act by elements of the form $(1,s)$ for $s \in \Gm$. We conclude that inside each $\Gm^{[n+1]}$-orbit in $\Gcal(\text{A},\text{B}_j)$ there is a $\Gm$-orbit parametrising subspaces $V^\prime$ satisfying the stronger condition:
\begin{enumerate}
\item[($\text{A}^{\!\star}$)] $V^\prime \cap \mathbb{V}(x_{n+1}) = V$.	
\end{enumerate}
We let $\Gcal(\text{A}^{\!\star},\text{B}_j) \subseteq \Gcal(\text{A},\text{B}_j)$ denote the locus consisting of subspaces satisfying ($\text{A}^{\!\star}$) and $(\text{B}_j)$. This consists of a union of $\Gm$-orbits, where the action is via the final coordinate. There is an identification
\[ \Gcal(\text{A}^{\!\star},\text{B}_j)/\Gm = \Gcal(\text{A},\text{B}_j)/\Gm^{[n+1]} \]
and Proposition \ref{prop: Gm n plus one orbit} thus reduces to:
\begin{proposition} \label{prop: Gm orbit} The locus $\Gcal(\text{A}^{\!\star},\text{B}_j)$ in fact consists of a single $\Gm$-orbit.\end{proposition}

\proofstep{Open locus $\Ugen$} \label{step: open locus}
We first fix a subspace $V^\prime$ satisfying ($\text{A}^{\!\star}$). This means that there exists
\[ \boxed{\ul{u}} \in \kfield^{[n]} \]
such that
\[ V^\prime = V + \kfield \cdot (\ul{u},1). \]
Note that $\ul{u}$ is uniquely defined up to adding elements of $V$. Conversely, any choice of $\ul{u}$ produces a $V^\prime$ satisfying ($\text{A}^{\!\star}$). 

However, $\ul{u}$ must be chosen generically in order to ensure that $V^\prime$ is coordinate-general. This produces an open inclusion
\[ \Gcal(\text{A}^{\!\star}) \subseteq \kfield^{[n]}/V \]
of the locus $\Gcal(\text{A}^{\!\star})$ parametrising coordinate-general $V^\prime$ satisfying ($\text{A}^{\!\star}$). This open subset does not include the origin $\ul{u}=\ul{0}$, since $V^\prime = V + \kfield \cdot (\ul{0},1)$ is not coordinate-general. Moreover it is invariant under the scaling action $\Gm \curvearrowright \kfield^{[n]}/V$ since this action multiplies each Pl\"ucker coordinate of~$V'$ by a character of~$\Gm$ (see the matrix $V^\prime$ in Step~\ref{step: equations for subsets}).
Thus the above inclusion descends to projective spaces: there is an open subset
\begin{equation} \label{eqn: Gm orbits of A bullet} \boxed{\Ugen} \subseteq \PP(\kfield^{[n]}/V) \end{equation}
parametrising choices of $[\ul{u}+V]$ such that $V^\prime \colonequals V + \kfield \cdot (\ul{u},1)$ is coordinate-general, and we have a natural identification:
\[ \Ugen = \Gcal(\text{A}^{\!\star})/\Gm. \]

\proofstep{Locally-closed loci $H_j$} \label{step: locally closed} 
We now impose condition ($\text{B}_j$) on $V^\prime = V + \kfield \cdot (\ul{u},1)$. We first calculate:
\begin{align*}
p_{I_j^\prime}(V^\prime) & = p_{I_j}(V) + \kfield \cdot (p_{I_j}(\ul{u}),1) \\
& = V_j + \kfield \cdot (p_{I_j}(\ul{u}),1).	
\end{align*}
On the other hand, the reasoning in Step~\ref{step: open locus} now applied to $V_j^\prime$ shows that there exists
\[ \boxed{\ul{w}_j} \in \kfield^{I_j} \]
such that
\[ V_j^\prime = \ul{t}_j V_j + \kfield \cdot (\ul{w}_j,1) \]
and the image of $\ul{w}_j$ in the quotient $\kfield^{I_j}/\ul{t}_j V_j$ is uniquely defined and nonzero. It follows that if we are to have $p_{I_j^\prime}(V^\prime) = \ul{t}_j^\prime V_j^\prime$ then $\ul{t}_j^\prime$ must take the form
\[ \ul{t}_j^\prime = (\ul{t}_j^{-1} ,s_j) \in \Gm^{I_j^\prime} = \Gm^{I_j} \times \Gm \]
for some $s_j \in \Gm$, where $\ul{t}_j$ is the group element described in Step~\ref{step: fixing target point}. This then gives:
\begin{align*} \ul{t}_j^\prime V_j^\prime & = V_j + \kfield \cdot (\ul{t}_j^{-1} \ul{w}_j, s_j) \\
& = V_j + \kfield \cdot (s_j^{-1} \ul{t}_j^{-1} \ul{w}_j,1).	
\end{align*}
Recall that $\ul{w}_j \in \kfield^{I_j}$ is only well-defined after passing to the quotient $\kfield^{I_j}/\ul{t}_j V_j$. We conclude that $V^\prime$ satisfies condition ($\text{B}_j$) if and only if the associated vector $\ul{u}$ is such that the images of $p_{I_j}(\ul{u})$ and $\ul{t}_j^{-1} \ul{w}_j$ in the quotient $\kfield^{I_j}/V_j$ lie in the same $\Gm$-orbit.

To unravel this condition, we first consider the coordinate projection:
\[ p_{I_j} \colon \kfield^{[n]} \to \kfield^{I_j}.\]
This sends $V \subseteq \kfield^{[n]}$ into $V_j \subseteq \kfield^{I_j}$ and thus descends to a linear map on the quotients:
\begin{equation} \label{eqn: linear projection on vector spaces} \kfield^{[n]}/V \to \kfield^{I_j}/V_j. \end{equation}
Consider the open set:
\[ \Ucal_j \colonequals \PP(\kfield^{[n]}/V) \setminus \PP((\kfield^{[n] \setminus I_j}+V)/V).\]
Then \eqref{eqn: linear projection on vector spaces} descends to a morphism:
\begin{equation} \label{eqn: linear projection on projective spaces} \Ucal_j \to \PP(\kfield^{I_j}/V_j).\end{equation}
The locus of $\Gm$-orbits which satisfy $(\text{B}_j)$ is then the preimage of $[\ul{t}_j^{-1} \ul{w}_j+V_j]$ under \eqref{eqn: linear projection on projective spaces}. We denote this:
\begin{equation} \label{eqn: Hj} \boxed{H_j} \subseteq \Ucal_j. \end{equation}
Since $\dim V_j = r$ and $|I_j| = r+2$, we in fact have $\PP(\kfield^{I_j}/V_j) \cong \PP^1$ and since the projection \eqref{eqn: linear projection on projective spaces} is linear it follows that the closure
\[ \ol{H}_j \subseteq \PP(\kfield^{[n]}/V) \]
is a hyperplane.

\subsection{Combining the conditions} \label{sec: dimension reduction proof combining conditions} Summarising, we have the open subset \eqref{eqn: Gm orbits of A bullet} parametrising $\Gm$-orbits of $V^\prime$ satisfying ($\text{A}^{\!\star}$) which are coordinate-general,
\[ \Ucal \subseteq \PP(\kfield^{[n]}/V), \]
and the locally-closed subset \eqref{eqn: Hj} parametrising $\Gm$-orbits of $V^\prime$ which satisfy ($\text{B}_j$),
\[ H_j \subseteq \PP(\kfield^{[n]}/V). \]
We now intersect these subsets, and show that for generic choices of starting data in Step~\ref{step: fixing target point}, the intersection is a single point, thus establishing Proposition \ref{prop: Gm orbit} and Theorem \ref{thm: dimension reduction}.

In Step~\ref{step: intersecting conditions} we set up the intersection problem. In Step~\ref{step: equations for subsets} we then write explicit equations for the above subsets. Finally, the difficult part is to establish the inclusion \eqref{eqn: inclusion of intersection of Hjbar}. We achieve this in Steps~\ref{step: swapping Plucker}~and~\ref{step: I vanishing}, by relating the Pl\"ucker coordinates of two matrices via Grassmannian duality, and then applying the surplus condition.

\proofstep{Intersecting the conditions} \label{step: intersecting conditions} To impose the conditions ($\text{A}^{\! \star}$) and ($\text{B}_j$) simultaneously, we simply form the intersection of the above sets. This gives the identity
\[ \Gcal(\text{A}^{\!\star},\text{B}_j)/\Gm = \Ucal \cap \bigcap_{j \in [k]} H_j .\]
However since points of $\ol{H}_j \setminus H_j$ corresponds to subspaces $V^\prime$ which are not coordinate-general, we in fact have\, $\Ucal \cap (\ol{H}_j \setminus H_j) = \emptyset$, and so
\begin{equation} \label{eqn: fibre as intersection of open and closed} \Gcal(\text{A}^{\!\star},\text{B}_j)/\Gm = \Ucal \cap \bigcap_{j \in [k]} \ol{H}_j .\end{equation}
The intersection of the $\ol{H}_j$ is an intersection of hyperplanes in projective space, of expected dimension zero. Consequently, it is either infinite or a singleton. It follows that $\Gcal(\text{A}^{\!\star},\text{B}_j)/\Gm$ is either infinite, or a singleton, or empty. 

However since $\Gcal(\text{A}^{\!\star},\text{B}_j)/\Gm$ is also a generic fibre of the map \eqref{eq:MapToFiberProduct} between varieties of the same dimension, it cannot be infinite. Therefore, it is either a singleton or empty. This is the key consequence of the preceding Steps~\ref{step: fibre counting}--\ref{step: intersecting conditions}. In the remaining Steps~\ref{step: equations for subsets}--\ref{step: I vanishing} we show that it is in fact nonempty, and therefore a singleton. The key result is:
\begin{proposition} \label{prop: inclusion of intersection of Hjbar} For generic choices of starting data in Step~\ref{step: fixing target point}, we have:
\begin{equation} \label{eqn: inclusion of intersection of Hjbar} \bigcap_{j \in [k]} \ol{H}_j \subseteq \Ucal. \end{equation}
\end{proposition}
Assuming Proposition \ref{prop: inclusion of intersection of Hjbar}, we obtain an identification
\[ \Gcal(\text{A}^{\!\star},\text{B}_j)/\Gm = \Ucal \cap \bigcap_{j \in [k]} \ol{H}_j = \bigcap_{j \in [k]} \ol{H}_j. \]
We have already observed that this cannot be infinite. Therefore it must be a singleton, which establishes Proposition \ref{prop: Gm orbit} and hence Theorem \ref{thm: dimension reduction}.

\begin{remark} \label{rmk: proof easier if conjecture known} Showing that $\Gcal(\text{A}^{\!\star},\text{B}_j)/\Gm$ is nonempty (and therefore a singleton) is straightforward if the positivity conjecture (Conjecture \ref{conj: positivity}) is assumed. Recall from \eqref{assumption count nonzero} that we can assume $d_{r,n}(\Ical) \neq 0$. By Proposition \ref{prop: nonzero implies surplus} this implies that $\Ical$ satisfies the surplus condition, and it follows that $\Ical^\prime$ also satisfies the surplus condition. Assuming the conjecture, we conclude that $d_{r+1,n+1}(\Ical^\prime) \neq 0$. Then \eqref{eq:MapToFiberProduct} must be dominant, since if it is not then neither is $\fgt_{\Ical^\prime}$. Thus, the generic fibre of \eqref{eq:MapToFiberProduct} is nonempty.
\end{remark}

It remains to prove Proposition \ref{prop: inclusion of intersection of Hjbar}. For this it is enough to work with the preimages, that is with the representatives $\ul{u} \in \kfield^{[n]}$ rather than the classes $[\ul{u}+V] \in \PP(\kfield^{[n]}/V)$. We adopt this perspective from here on.

\proofstep{Equations for the subsets} \label{step: equations for subsets}
We begin by giving precise equations for the subsets $\Ucal$ and $\ol{H}_j$. We begin with $\Ucal$, defined in Step~\ref{step: open locus} above. Fix a basis $\ul{v}_1,\ldots,\ul{v}_r \in \kfield^{[n]}$ for the subspace $V$ so that $V$ is the column span of the following $n \times r$ matrix, which we denote by the same symbol:
\[ 
\boxed{V} = \left( \begin{array}{ccc}
\!\! \vline & & \vline \!\! \\
\! \ul{v}_1 & \cdots & \ul{v}_r \!\!\! \\
\!\! \vline & & \vline \!\!
\end{array}
\right).
\]
Then given $\ul{u}=(u_1,\ldots,u_n) \in \kfield^{[n]}$ the subspace $V^\prime = V + \kfield \cdot (\ul{u},1)$ is the column span of the following $(n\!+\!1) \times (r\!+\!1)$ matrix:
\[ 
\boxed{V^\prime} = \left( \begin{array}{cccc}
\!\! \vline & & \vline \!\! & u_1 \!\! \\
\! \ul{v}_1 & \cdots & \ul{v}_r \!\!\! & \vdots \! \\
\!\! \vline & & \vline \!\! & u_n \!\! \\
0 & \cdots & 0 & 1
\end{array}
\right)
= \left( \begin{array}{cccc}
& & & \\
& V & & \ul{u} \\
& & & \\
0 & \cdots & 0 & 1 	
\end{array}
\right).
\]
The open subset $\Ucal$ is the locus where $V^\prime$ is coordinate-general, which is precisely the locus where all the $(r\!+\!1) \times (r\!+\!1)$ minors of the above matrix do not vanish. Here the $u_1,\ldots,u_n$ are coordinates on $\kfield^{[n]}$ while the entries of $V$ are fixed scalars.

Such a minor is produced by selecting $r\!+\!1$ rows, and there are two cases to consider. If the final row is not selected, then the resulting inequation is linear in the $u_i$. If the final row is selected, then the resulting equation does not involve the $u_i$ at all, and instead reduces to one of the inequations guaranteeing that $V$ is coordinate-general.

We only need consider the first case. Given a subset $S \in \binom{[n]}{r+1}$ we let
\[ F_S(\ul{u}) \colonequals \det  \left( V|_S \ \ \ul{u}|_S \right) \]
where the restriction denotes restriction of rows. Then the open set $\Ucal$ is given by:
\begin{equation} \label{eqn: equation for U} \Ucal = \left\{ F_S(\ul{u}) \neq 0 \, \colon \text{for all } S \in \binom{[n]}{r\!+\!1} \right\}.\end{equation}

We next describe $H_j$ and $\ol{H}_j$. The condition for $\ul{u}$ to belong to $H_j$ is that there exists $s \in \Gm$ such that
\[ p_{I_j}(\ul{u}) \in V_j + s \ul{t}_j^{-1} \ul{w}_j .\]
Passing to the closure $\ol{H}_j$ means allowing $s \in \kfield$ and so the condition becomes linear:
\[ p_{I_j}(\ul{u}) \in V_j + \kfield \cdot \ul{t}_j^{-1} \ul{w}_j . \]
This is equivalent to the determinant of an $(r+2) \times (r+2)$ matrix being zero:
\[ \det \left( V|_{I_j} \ \ \ul{t}_j^{-1} \ul{w}_j \ \ \ul{u}|_{I_j} \right) = 0.\]
Expanding along the final column, this gives
\begin{equation} \label{eqn: for Hjbar} \ol{H}_j = \left\{ \sum_{i \in I_j} \sgn(i\!\in\!I_j)\, F_{I_j \setminus i}(\ul{t}_j^{-1} \ul{w}_j) \cdot u_i = 0 \right\}, \end{equation}
where $\sgn(i\!\in\!I_j) = 1$ (respectively $-1$) if $i$ appears in an odd (respectively even) position in $I_j$. Intersecting the \eqref{eqn: for Hjbar} for all $j \in [k]$ we obtain a system of linear equations in the variables $u_1,\ldots,u_n$. We encode this system in a single $k \times n$ matrix $A$ as follows:
\begin{equation} \label{eqn: defn of A} \boxed{A_{ji}} \colonequals \begin{cases} \sgn(i\!\in\!I_j)\, F_{I_j \setminus i}(\ul{t}_j^{-1} \ul{w}_j) & \text{if $i \in I_j$}, \\ 0 & \text{if $i \not\in I_j$}. \end{cases} \end{equation}
Then we have:
\[ \bigcap_{j \in [k]} \ol{H}_j = \left\{ A \ul{u} = \ul{0} \right\}\!. \]

\proofstep{Swapping the Pl\"ucker coordinates} \label{step: swapping Plucker} Recall from Step~\ref{step: intersecting conditions} that Theorem \ref{thm: dimension reduction} has been reduced to Proposition \ref{prop: inclusion of intersection of Hjbar}, which asserts that the inclusion \eqref{eqn: inclusion of intersection of Hjbar} holds for generic choices of starting data in Step~\ref{step: fixing target point}.

\begin{proposition} \label{prop: A nonvanishing Plucker enough} Suppose that $A$ has all Pl\"ucker coordinates nonvanishing. Then the inclusion \eqref{eqn: inclusion of intersection of Hjbar} holds. \end{proposition}

\begin{proof} Fix an element
\[ \ul{u} \in \bigcap_{j \in [k]} \ol{H}_j.\] 
By the previous step, this means precisely that
\[ A \ul{u} = \ul{0}.\]
We must show that $\ul{u} \in \Ucal$. Note that given any $\ul{v} \in \kfield^{[n]}$ we have:
\begin{align*} A \ul{v} = \ul{0} \Leftrightarrow \ & \det \left( V|_{I_j} \ \ \ul{t}_j^{-1} \ul{w}_j \ \ \ul{v}|_{I_j} \right) = 0 \qquad \text{ for all $j \in [k]$} \\
\Leftrightarrow \ & \ \, \ul{v}|_{I_j} \in V|_{I_j} + \kfield \cdot \ul{t}_j^{-1} \ul{w}_j \qquad \quad \, \text{ \, for all $j \in [k]$}.
\end{align*}
In particular if $\ul{v} \in V$ then $A \ul{v} = \ul{0}$. Taking the basis elements $\ul{v}_1,\ldots,\ul{v}_r$ from Step~\ref{step: equations for subsets}, if we define the $n \times (r+1)$ matrix
\[ \boxed{B} \colonequals \left( \begin{array}{cccc}
\!\! \vline & & \vline \!\! & \vline \!\! \\
\! \ul{v}_1 & \cdots & \ul{v}_r \!\!\! & \ul{u} \! \\
\!\! \vline & & \vline \!\! & \vline \!\! \end{array} \right) = \left( V \ \ \ul{u} \right)\]
then we have:
\[ AB = 0 .\]
Since $k+(r+1)=n$ the matrices $A$ and $B$ have complementary sizes. Since $A$ has all Pl\"ucker coordinates nonvanishing it in particular has full rank, while $B$ has full rank since $\ul{u} \not\in V$. Therefore they fit into a short exact sequence:
\[ 0 \to \kfield^{r+1} \xrightarrow{B} \kfield^n \xrightarrow{A} \kfield^{k} \to 0. \]
It follows from Grassmannian duality that their Pl\"ucker coordinates coincide, up to reindexing and signs (see e.g.\ \cite[Proposition~3.1]{Magyar}). Since $A$ has all Pl\"ucker coordinates nonvanishing, it follows that $B$ does as well. But by \eqref{eqn: equation for U} this is precisely the statement that $\ul{u} \in \Ucal$.\end{proof}

By Proposition \ref{prop: A nonvanishing Plucker enough}, it remains to prove:

\begin{proposition} \label{prop: A has nonvanishing Plucker} For generic choices of starting data in Step~\ref{step: fixing target point}, the matrix $A$ has all Pl\"ucker coordinates nonvanishing.\end{proposition}

To prove this, first recall that the choice of starting data in Step~\ref{step: fixing target point} may be obtained by choosing the following data independently:
\begin{itemize}
\item An $n \times r$ matrix $V$ (see Step~\ref{step: equations for subsets}).
\item A torus element $\ul{t}_j \in \Gm^{I_j}$ for each $j \in [k]$ (see Steps~\ref{step: fixing target point}~and~\ref{step: locally closed}).
\item A vector $\ul{w}_j \in \kfield^{I_j}$ for each $j \in [k]$ (see Step~\ref{step: locally closed}).
\end{itemize}
Given these, the subspace $V$ is defined as the column span of the matrix $V$, the subspace $V_j$ is the column span of the matrix $V|_{I_j}$, and the subspace $V_j^\prime$ is defined by:
\[ V_j^\prime \colonequals \ul{t}_j V_j + \kfield \cdot (\ul{w}_j,1). \]
The above data is subject to open conditions which ensure that the associated subspaces are coordinate-general. We now show that generic such choices produce a matrix $A$ with all Pl\"ucker coordinates nonvanishing, verifying Proposition \ref{prop: A has nonvanishing Plucker}.

\proofstep{Matrices with $\Ical$-vanishing} \label{step: I vanishing}

\begin{definition} A $k \times n$ matrix $C$ has \textbf{$\Ical$-vanishing} if and only if:
\[ i \in I_j \Rightarrow C_{ji} = 0. \]	
\end{definition}
The matrix $A$ defined in \eqref{eqn: defn of A} has $\Ical$-vanishing by definition. In this final step, we establish the following two results concerning such matrices.

\begin{lemma} \label{lem: general I vanishing matrix is coordinate-general} A general matrix with $\Ical$-vanishing has all Pl\"ucker coordinates nonvanishing.
\end{lemma}

\begin{lemma} \label{lem: can achieve a general matrix with I vanishing} Given a general matrix $C$ with $\Ical$-vanishing, there exists a choice of $V, \ul{t}_j, \ul{w}_j$ such that the associated matrix $A$ defined in \eqref{eqn: defn of A} is equal to $C$.	
\end{lemma}

\begin{proof}[Proof~of~Proposition~\ref{prop: A has nonvanishing Plucker}] Assume Lemmas~\ref{lem: general I vanishing matrix is coordinate-general}~and~\ref{lem: can achieve a general matrix with I vanishing}. There is a polynomial map of affine spaces
\begin{equation} \label{eqn: map from V and wj space to A space} \Aaff_{V,\ul{t}_j,\ul{w}_j} \to \Aaff_A \end{equation}
sending choices of $V,\ul{t}_j,\ul{w}_j$ to the associated matrix $A$ with $\Ical$-vanishing. Lemma \ref{lem: can achieve a general matrix with I vanishing} states that this map is dominant, while Lemma \ref{lem: general I vanishing matrix is coordinate-general} states that there is a nonempty open subset of the target consisting of matrices with all Pl\"ucker coordinates nonvanishing. Since the map is dominant, the preimage of this open subset is a nonempty open subset of the domain, and is therefore dense. Thus for a generic choice of $V,\ul{t}_j,\ul{w}_j$ the associated matrix $A$ has all Pl\"ucker coordinates nonvanishing, as required.
\end{proof}

It remains to prove Lemmas~\ref{lem: general I vanishing matrix is coordinate-general}~and~\ref{lem: can achieve a general matrix with I vanishing}.

\begin{proof}[Proof of Lemma \ref{lem: general I vanishing matrix is coordinate-general}] We introduce variables
\[ x_{ji} \]
for $j \in [k]$ and $i \in I_j$. Then the affine space $\Aaff_{\ul{x}}$ parametrises $k \times n$ matrices with $\Ical$-vanishing, and there is a universal such matrix
\[ C(\ul{x}) \]
with $C(\ul{x})_{ji} \colonequals x_{ji}$. Given $S \in \binom{[n]}{r+1}$ we consider the associated Pl\"ucker coordinate:
\begin{equation} \label{eqn: Plucker coordinate of general I vanishing matrix} \det(C(\ul{x})|_{[n]\setminus S}).\end{equation}
This is a degree $k$ homogeneous polynomial in the $x_{ji}$ which we now describe explicitly.

Recall from \eqref{assumption count nonzero} that we are assuming $d_{r,n}(\Ical) \neq 0$. It follows from Proposition \ref{prop: nonzero implies surplus} and Theorem \ref{thm: surplus iff matching} that there exists a perfect matching on $\Gamma(\Ical)\!\setminus\!S$, namely a bijection
\[ \varphi \colon [k] \to [n] \setminus S \]
such that $\varphi(j) \in I_j$ for all $j \in [k]$. There is then a term in \eqref{eqn: Plucker coordinate of general I vanishing matrix} of the form
\[ \pm \prod_{j \in [k]} x_{j,\varphi(j)} \]
and in fact, we have
\[ \det(C(\ul{x})|_{[n] \setminus S}) = \sum_{\varphi} \pm \prod_{j \in [k]} x_{j,\varphi(j)} \]
where the sum is over all perfect matchings of $\Gamma(\Ical)\!\setminus\!S$, and crucially the sum is nonempty. Each term in the sum is a monomial of degree $k$, and these monomials are all distinct and therefore do not cancel. It follows that $\det(C(\ul{x})|_{[n]\setminus S})$ is a nonzero polynomial in the $x_{ji}$ and so the locus
\[ \left\{ \det(C(\ul{x})|_{[n] \setminus S}) \neq 0 \right\} \subseteq \Aaff_{\ul{x}} \]
is a nonempty open set. Intersecting these for all $S$ produces a nonempty open set, which is therefore dense.
\end{proof}

\begin{proof}[Proof~of~Lemma~\ref{lem: can achieve a general matrix with I vanishing}] Fix a generic $k \times n$ matrix $C$ with $\Ical$-vanishing. We must find starting data $V,\ul{t}_j,\ul{w}_j$ such the associated matrix $A$ is equal to $C$.

Note that it is not necessary to ensure that the starting data is coordinate-general, because once we know that the map \eqref{eqn: map from V and wj space to A space} is dominant, it will remain dominant after restricting to the open subset of coordinate-general inputs.

We first determine $V$. We view $C$ as a map:
\[ C \colon \kfield^n \to \kfield^k. \]
The kernel has dimension at least $n-k=r+1$. We can thus choose $r$ linearly independent elements $\ul{v}_1,\ldots,\ul{v}_r$ of the kernel, and let $V$ be their column matrix, so that
\[ C V = \ul{0}.\]
Focusing on the $j$th row $C_j$, the above equation is equivalent to the vanishing of the inner product
\[ C_j \cdot \ul{v}_i = 0 \]
for $i=1,\ldots,r$. Since $C_{ji} = 0$ for $i \not\in I_j$, the entry $v_{ij}$ only appears in this equation if $i \in I_j$. This equation is thus equivalent to:
\[ C_j|_{I_j} V|_{I_j} = \ul{0}. \]
Now choose $\ul{t}_j \in \Gm^{I_j}$ arbitrarily, and suppose we choose $\ul{w}_j \in \kfield^{I_j}$ such that:
\begin{equation} \label{eqn: equation for wj} C_j|_{I_j} \cdot (\ul{t}_j^{-1} \ul{w}_j) = 0. \end{equation}
This gives:
\[ C_j|_{I_j} \left( V|_{I_j} \ \ \ul{t}_j^{-1} \ul{w}_j \right) = \ul{0}.\]
The $1 \times (r+2)$ matrix $C_j|_{I_j}$ has full rank (that is, it is not the zero vector) since $C$ was chosen generically. On the other hand the $(r+2) \times (r+1)$ matrix $(V|_{I_j} \ \ \ul{t}_j^{-1} \ul{w}_j)$ has full rank if and only if
\begin{equation} \label{eqn: inequation for wj} \ul{t}_j^{-1} \ul{w}_j \not\in V|_{I_j}.\end{equation}
Assuming this, since the matrices have complementary sizes, they have the same Pl\"ucker coordinates up to reindexing and signs (as in Step~\ref{step: swapping Plucker} above). The Pl\"ucker coordinates are indexed by a choice of $i \in I_j$, and for each we obtain the identity:
\[ \sgn(i\!\in\!I_j) \det( V|_{I_j \setminus i} \ \ (\ul{t}_j^{-1} \ul{w}_j)|_{I_j \setminus i} ) = C_{ji}.\]
The left-hand side is precisely the $A_{ji}$ as induced by our choice of $V, \ul{t}_j, \ul{w}_j$. We thus obtain $A=C$ as required, assuming we can choose $\ul{w}_j$ such that \eqref{eqn: equation for wj} and \eqref{eqn: inequation for wj} hold. 

But $\Ker C_j \subseteq \kfield^{I_j}$ is a hyperplane, whereas $V|_{I_j} \subseteq \kfield^{I_j}$ is $r$-dimensional, so has codimension~$2$. Therefore:
\[ \Ker C_j \setminus V|_{I_j} \neq \emptyset. \]
We choose any element of this set, then act by $\ul{t}_j$ to obtain the $\ul{w}_j$ which produces the desired $A$.
\end{proof}

This completes the proof of Theorem \ref{thm: dimension reduction}.

\begin{remark} Throughout this paper we assume $r \geqslant 2$, but it is worth reflecting on what the case $r=1$ should mean. In this setting, $X(1,[n])$ parametrises point configurations in $\PP^0 = \pt$ which are linearly general. There is only one such configuration, and trivially it is linearly general, so:
\[ X(1,[n]) = \pt .\]
Less trivially, the description as a Grassmannian quotient (Section \ref{sec: Grassmannians}) still holds:
\[ X(1,[n]) = \Gr^\circ(1,\kfield^{[n]})/\Gm^{[n]} = \PP(\Gm^{[n]})/\Gm^{[n]} = \pt. \]
A projective configuration count is determined by a choice of $k=n-r-1=n-2$ subsets of size $r+2=3$, but we always have $d_{1,n}(\Ical) = 1$ since $\fgt_\Ical$ is the map from a point to a point.

Crucially however, Theorem \ref{thm: dimension reduction} fails for $r=1$. An example with $n=4$ is given by $\Ical = (234,234)$. Then $d_{1,n}(\Ical)=1$, but $d_{2,n+1}(\Ical^\prime) = 0$ since $\Ical^\prime = (2345,2345)$ constraints the same cross-ratio twice.

While such examples are very simple, the failure of the above proof is quite subtle. The Grassmannian constructions carry through and are not even entirely vacuous. The problem, in fact, does not appear until the proof of Lemma~\ref{lem: general I vanishing matrix is coordinate-general}. Here we use the fact that if $d_{r,n}(\Ical) \neq 0$ then the graph $\Gamma(\Ical)\!\setminus\!S$ must admit a perfect matching. This fails when $r=1$, as the above example with $S=34$ attests. This also shows that the analogue of Proposition \ref{prop: nonzero implies surplus} fails when $r=1$.
\end{remark}

\subsection{Reducing the upper bound} \label{sec: combining theorems} Our main Theorems~\ref{thm: upper bound} and \ref{thm: dimension reduction} can be combined. Suppose as in Theorem \ref{thm: dimension reduction} we are given a tuple
\[ \Ical^\prime = (I_1^\prime,\ldots,I_k^\prime) \in \binom{[n\!+\!1]}{(r\!+\!1)+2}^k \]
such that $n+1 \in I_j^\prime$ for all $j \in [k]$. Then setting $I_j \colonequals I_j^\prime \setminus \{ n+1 \}$ and $\Ical = (I_1,\ldots,I_k)$ we obtain:
\begin{equation} \label{eqn: equality of configuration counts reducing upper bound} d_{r+1,n+1}(\Ical^\prime) = d_{r,n}(\Ical). \end{equation}
We now compare the upper bounds for these counts given by Theorem \ref{thm: upper bound}. Fix a subset $S^\prime \subseteq [n+1]$ of size $(r+1)+2$ and assume that $n+1 \in S^\prime$. Let $S \colonequals S^\prime \setminus \{ n+1 \}$. We have an equality of pruned configuration graphs:
\[ \Gamma(\Ical^\prime)\setminus S^\prime = \Gamma(\Ical)\!\setminus\!S.\]
Then Theorem \ref{thm: upper bound} implies that the configuration count \eqref{eqn: equality of configuration counts reducing upper bound} is bounded above by both the number of $r$-weighted transversals and the number of $(r-1)$-weighted transversals of this graph. The following result, which we prove below, shows that the latter bound is more effective:

\begin{proposition} \label{prop: number of weighted transversals monotonic} Fix a bipartite graph $\Gamma$ and a positive integer $m$. Then the number of $m$-weighted transversals of $\Gamma$ is less than or equal to the number of $(m+1)$-weighted transversals of $\Gamma$.\end{proposition}
Thus when bounding a projective configuration count, we should\footnote{There is a caveat: above we have assumed that before we drop dimension, the set $S^\prime$ includes the marking $n+1$ to be forgotten. If we remove this assumption we can produce graphs different from those which appear after dropping dimension, and a priori these could produce more efficient upper bounds, though we know of no such case.} first reduce dimensions as far as possible using Theorem \ref{thm: dimension reduction}, and only then calculate the upper bound in Theorem \ref{thm: upper bound}.

The proof of Proposition \ref{prop: number of weighted transversals monotonic} has two ingredients: addition of transversals and Hall's marriage theorem. We begin with addition of transversals. Recall that $T_m(\Gamma)$ denotes the set of $m$-weighted transversals of $\Gamma$. Given weighted transversals $T_1 \in T_{m_1}(\Gamma)$ and $T_2 \in T_{m_2}(\Gamma)$ we produce a weighted transversal
\[ T_1 + T_2 \in T_{m_1+m_2}(\Gamma) \]
by summing the two weights associated to each edge. This operation will allow us to relate transversals of different weights.

We now discuss Hall's marriage theorem. Given a subset $A \subseteq V(\Gamma)$ of the left vertices of the bipartite graph, the \textbf{neighbourhood} $N(A) \subseteq V(\Gamma)$ consists of those right vertices which are adjacent to a vertex in $A$. We then have:
\begin{theorem}[Hall's marriage theorem \cite{HallMarriage}] \label{thm: Hall} The bipartite graph $\Gamma$ admits a $1$-weighted transversal (i.e.\ a perfect matching) if and only if
\[ |N(A)| \geqslant |A| \]
for all subsets $A \subseteq V(\Gamma)$ of the left vertices.
\end{theorem}

\begin{lemma} \label{lem: no transversals} Fix positive integers $m_1,m_2$. Then $\Gamma$ admits an $m_1$-weighted transversal if and only if $\Gamma$ admits an $m_2$-weighted transversal.
\end{lemma}

\begin{proof} It is sufficient to prove the claim with $m_2=1$. Suppose first that $\Gamma$ admits a $1$-weighted transversal $T$. Then $m_1 T$ is an $m_1$-weighted transversal.

Suppose conversely that $\Gamma$ admits an $m_1$-weighted transversal. Fix a subset $A \subseteq V(\Gamma)$ of the left vertices. Summing the weights of the edges adjacent to $A$ gives:
\[ \sum_{v \in A} \sum_{e \ni v} m_e = m_1 |A|. \]
On the other hand, this set of edges is a subset of the set of edges adjacent to $N(A)$. We thus obtain
\[ m_1 |N(A)| = \sum_{v \in N(A)} \sum_{e \ni v} m_e \geqslant \sum_{v \in A} \sum_{e \ni v} m_e = m_1 |A| \]
and so $|N(A)| \geqslant |A|$. By Hall's marriage theorem (Theorem \ref{thm: Hall}), we conclude that $\Gamma$ admits a $1$-weighted transversal.	
\end{proof}

\begin{proof}[Proof of Proposition \ref{prop: number of weighted transversals monotonic}] We first examine the case where $\Gamma$ admits a $1$-weighted transversal $T_1$. Consider the map:
\begin{align*} T_m(\Gamma) & \to T_{m+1}(\Gamma) \\
T & \mapsto T+T_1.	
\end{align*}
This map is clearly injective, which establishes the desired inequality. On the other hand if $\Gamma$ does not admit a $1$-weighted transversal, then by Lemma \ref{lem: no transversals} we have
\[ |T_m(\Gamma)| = |T_{m+1}(\Gamma)| = 0 \]
and so the inequality is trivially satisfied.	
\end{proof}

\begin{example} Fix $r=3, n=7$ so that $k=n-r-1=3$ and consider the input data:
\[ \Ical^\prime = ( \{1,2,3,4,7\},  \{3,4,5,6,7\}, \{1,2,5,6,7\} ). \]
Since $7$ belongs to every subset, we may remove it and produce the $r=2, n=6$ input data:
\[ \Ical = ( \{1,2,3,4\}, \{3,4,5,6\}, \{1,2,5,6\} ).\]
The corresponding configuration count was computed in \cite[Example~1.3]{Silversmith2022}. We have:
\begin{equation} \label{eqn: dimension reduction example count} d_{3,7}(\Ical^\prime) = d_{2,6}(\Ical) = 2.\end{equation}
We now investigate the upper bounds. Taking $S^\prime = \{4,5,6,7\}$ and $S=\{4,5,6\}$ ensures that the pruned configuration graphs agree. We let $\Gamma = \Gamma(\Ical^\prime)\setminus S^\prime = \Gamma(\Ical)\!\setminus\!S$ denote this graph:
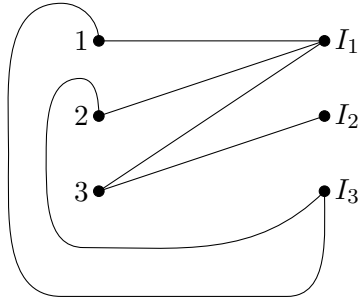
\[
\begin{tikzpicture}
    \draw[fill=black] (0,0) circle[radius=2pt];
    \draw (0,0) node[left]{$1$};

    \draw[fill=black] (0,-1) circle[radius=2pt];
     \draw (0,-1) node[left]{$2$};
     
      \draw[fill=black] (0,-2) circle[radius=2pt];
     \draw (0,-2) node[left]{$3$};

    \draw[fill=black] (3,0) circle[radius=2pt];
    \draw (3,0) node[right]{$I_1$};
    
    \draw[fill=black] (3,-1) circle[radius=2pt];
    \draw (3,-1) node[right]{$I_2$};

    \draw[fill=black] (3,-2) circle[radius=2pt];
    \draw (3,-2) node[right]{$I_3$};

    \draw (0,0) -- (3,0);
    \draw (0,0) to [out=90,in=0] (-0.5,0.5) to [out=180,in=90] (-1.2,-1.5) to [out=270,in=180] (-0.5,-3.4) to [out=0,in=180] (2.5,-3.4) to [out=0,in=270] (3,-2);
    \draw (0,-1) -- (3,0);
    \draw (0,-1) to [out=90,in=0] (-0.25,-0.5) to [out=180,in=90] (-0.7,-1.5) to [out=270,in=180] (-0.2,-2.75) to [out=0, in=225] (3,-2);
    \draw (0,-2) -- (3,0);
    \draw (0,-2) -- (3,-1);
\end{tikzpicture}
\]
By Theorem \ref{thm: upper bound} the configuration count \eqref{eqn: dimension reduction example count} is bounded above by both the number of $1$-weighted transversals and the number of $2$-weighted transversals of $\Gamma$. There are two $1$-weighted transversals
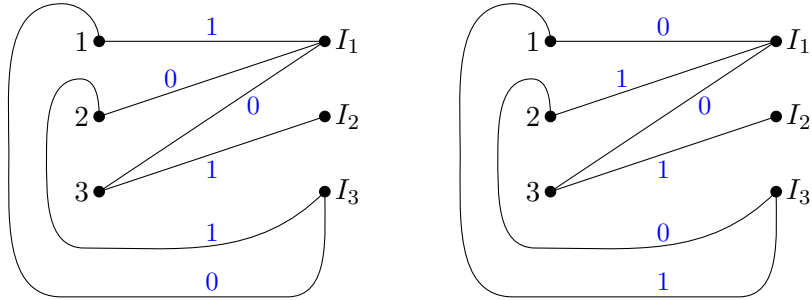
\[
\begin{tikzpicture}
    \draw[fill=black] (0,0) circle[radius=2pt];
    \draw (0,0) node[left]{$1$};

    \draw[fill=black] (0,-1) circle[radius=2pt];
     \draw (0,-1) node[left]{$2$};
     
      \draw[fill=black] (0,-2) circle[radius=2pt];
     \draw (0,-2) node[left]{$3$};

    \draw[fill=black] (3,0) circle[radius=2pt];
    \draw (3,0) node[right]{$I_1$};
    
    \draw[fill=black] (3,-1) circle[radius=2pt];
    \draw (3,-1) node[right]{$I_2$};

    \draw[fill=black] (3,-2) circle[radius=2pt];
    \draw (3,-2) node[right]{$I_3$};

    \draw (0,0) -- (3,0);
    \draw (1.5,-0.05) node[above,blue]{\small$1$};

    \draw (0,0) to [out=90,in=0] (-0.5,0.5) to [out=180,in=90] (-1.2,-1.5) to [out=270,in=180] (-0.5,-3.4) to [out=0,in=180] (2.5,-3.4) to [out=0,in=270] (3,-2);
    \draw (1.5,-3.2) node[blue]{\small$0$};
    
    \draw (0,-1) -- (3,0);
    \draw (0.95,-0.5) node[blue]{\small$0$};
    
    \draw (0,-1) to [out=90,in=0] (-0.25,-0.5) to [out=180,in=90] (-0.7,-1.5) to [out=270,in=180] (-0.2,-2.75) to [out=0, in=225] (3,-2);
    \draw (1.5,-2.55) node[blue]{\small$1$};

    \draw (0,-2) -- (3,0);
    \draw (2.05,-0.85) node[blue]{\small$0$};
    
    \draw (0,-2) -- (3,-1);
    \draw (1.5,-1.7) node[blue]{\small$1$};
    
    \draw[fill=black] (6,0) circle[radius=2pt];
    \draw (6,0) node[left]{$1$};

    \draw[fill=black] (6,-1) circle[radius=2pt];
     \draw (6,-1) node[left]{$2$};
     
      \draw[fill=black] (6,-2) circle[radius=2pt];
     \draw (6,-2) node[left]{$3$};

    \draw[fill=black] (9,0) circle[radius=2pt];
    \draw (9,0) node[right]{$I_1$};
    
    \draw[fill=black] (9,-1) circle[radius=2pt];
    \draw (9,-1) node[right]{$I_2$};

    \draw[fill=black] (9,-2) circle[radius=2pt];
    \draw (9,-2) node[right]{$I_3$};

    \draw (6,0) -- (9,0);
    \draw (7.5,-0.05) node[above,blue]{\small$0$};

    \draw (6,0) to [out=90,in=0] (5.5,0.5) to [out=180,in=90] (4.8,-1.5) to [out=270,in=180] (5.5,-3.4) to [out=0,in=180] (8.5,-3.4) to [out=0,in=270] (9,-2);
    \draw (7.5,-3.2) node[blue]{\small$1$};
    
    \draw (6,-1) -- (9,0);
    \draw (6.95,-0.5) node[blue]{\small$1$};
    
    \draw (6,-1) to [out=90,in=0] (5.75,-0.5) to [out=180,in=90] (5.3,-1.5) to [out=270,in=180] (5.8,-2.75) to [out=0, in=225] (9,-2);
    \draw (7.5,-2.55) node[blue]{\small$0$};

    \draw (6,-2) -- (9,0);
    \draw (8.05,-0.85) node[blue]{\small$0$};
    
    \draw (6,-2) -- (9,-1);
    \draw (7.5,-1.7) node[blue]{\small$1$};
    
\end{tikzpicture}
\]
whereas there are three $2$-weighted transversals:
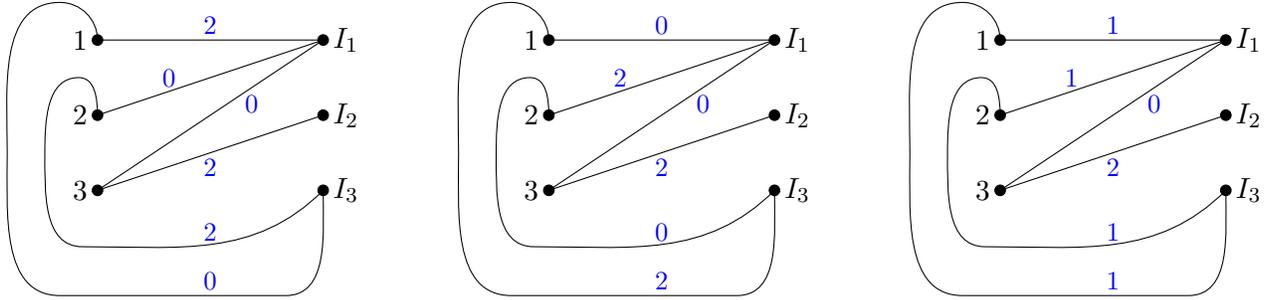
\[
\begin{tikzpicture}
    \draw[fill=black] (0,0) circle[radius=2pt];
    \draw (0,0) node[left]{$1$};

    \draw[fill=black] (0,-1) circle[radius=2pt];
     \draw (0,-1) node[left]{$2$};
     
      \draw[fill=black] (0,-2) circle[radius=2pt];
     \draw (0,-2) node[left]{$3$};

    \draw[fill=black] (3,0) circle[radius=2pt];
    \draw (3,0) node[right]{$I_1$};
    
    \draw[fill=black] (3,-1) circle[radius=2pt];
    \draw (3,-1) node[right]{$I_2$};

    \draw[fill=black] (3,-2) circle[radius=2pt];
    \draw (3,-2) node[right]{$I_3$};

    \draw (0,0) -- (3,0);
    \draw (1.5,-0.05) node[above,blue]{\small$2$};

    \draw (0,0) to [out=90,in=0] (-0.5,0.5) to [out=180,in=90] (-1.2,-1.5) to [out=270,in=180] (-0.5,-3.4) to [out=0,in=180] (2.5,-3.4) to [out=0,in=270] (3,-2);
    \draw (1.5,-3.2) node[blue]{\small$0$};
    
    \draw (0,-1) -- (3,0);
    \draw (0.95,-0.5) node[blue]{\small$0$};
    
    \draw (0,-1) to [out=90,in=0] (-0.25,-0.5) to [out=180,in=90] (-0.7,-1.5) to [out=270,in=180] (-0.2,-2.75) to [out=0, in=225] (3,-2);
    \draw (1.5,-2.55) node[blue]{\small$2$};

    \draw (0,-2) -- (3,0);
    \draw (2.05,-0.85) node[blue]{\small$0$};
    
    \draw (0,-2) -- (3,-1);
    \draw (1.5,-1.7) node[blue]{\small$2$};
    
    \draw[fill=black] (6,0) circle[radius=2pt];
    \draw (6,0) node[left]{$1$};

    \draw[fill=black] (6,-1) circle[radius=2pt];
     \draw (6,-1) node[left]{$2$};
     
      \draw[fill=black] (6,-2) circle[radius=2pt];
     \draw (6,-2) node[left]{$3$};

    \draw[fill=black] (9,0) circle[radius=2pt];
    \draw (9,0) node[right]{$I_1$};
    
    \draw[fill=black] (9,-1) circle[radius=2pt];
    \draw (9,-1) node[right]{$I_2$};

    \draw[fill=black] (9,-2) circle[radius=2pt];
    \draw (9,-2) node[right]{$I_3$};

    \draw (6,0) -- (9,0);
    \draw (7.5,-0.05) node[above,blue]{\small$0$};

    \draw (6,0) to [out=90,in=0] (5.5,0.5) to [out=180,in=90] (4.8,-1.5) to [out=270,in=180] (5.5,-3.4) to [out=0,in=180] (8.5,-3.4) to [out=0,in=270] (9,-2);
    \draw (7.5,-3.2) node[blue]{\small$2$};
    
    \draw (6,-1) -- (9,0);
    \draw (6.95,-0.5) node[blue]{\small$2$};
    
    \draw (6,-1) to [out=90,in=0] (5.75,-0.5) to [out=180,in=90] (5.3,-1.5) to [out=270,in=180] (5.8,-2.75) to [out=0, in=225] (9,-2);
    \draw (7.5,-2.55) node[blue]{\small$0$};

    \draw (6,-2) -- (9,0);
    \draw (8.05,-0.85) node[blue]{\small$0$};
    
    \draw (6,-2) -- (9,-1);
    \draw (7.5,-1.7) node[blue]{\small$2$};
    
    \draw[fill=black] (12,0) circle[radius=2pt];
    \draw (12,0) node[left]{$1$};

    \draw[fill=black] (12,-1) circle[radius=2pt];
     \draw (12,-1) node[left]{$2$};
     
      \draw[fill=black] (12,-2) circle[radius=2pt];
     \draw (12,-2) node[left]{$3$};

    \draw[fill=black] (15,0) circle[radius=2pt];
    \draw (15,0) node[right]{$I_1$};
    
    \draw[fill=black] (15,-1) circle[radius=2pt];
    \draw (15,-1) node[right]{$I_2$};

    \draw[fill=black] (15,-2) circle[radius=2pt];
    \draw (15,-2) node[right]{$I_3$};

    \draw (12,0) -- (15,0);
    \draw (13.5,-0.05) node[above,blue]{\small$1$};

    \draw (12,0) to [out=90,in=0] (11.5,0.5) to [out=180,in=90] (10.8,-1.5) to [out=270,in=180] (11.5,-3.4) to [out=0,in=180] (14.5,-3.4) to [out=0,in=270] (15,-2);
    \draw (13.5,-3.2) node[blue]{\small$1$};
    
    \draw (12,-1) -- (15,0);
    \draw (12.95,-0.5) node[blue]{\small$1$};
    
    \draw (12,-1) to [out=90,in=0] (11.75,-0.5) to [out=180,in=90] (11.3,-1.5) to [out=270,in=180] (11.8,-2.75) to [out=0, in=225] (15,-2);
    \draw (13.5,-2.55) node[blue]{\small$1$};

    \draw (12,-2) -- (15,0);
    \draw (14.05,-0.85) node[blue]{\small$0$};
    
    \draw (12,-2) -- (15,-1);
    \draw (13.5,-1.7) node[blue]{\small$2$};
    
\end{tikzpicture}
\]
This demonstrates that the inequality in Proposition \ref{prop: number of weighted transversals monotonic} can be strict. The $1$-weighted transversals give a more efficient (in this case, sharp) upper bound. In fact, \emph{no} choice of $S'$ gives a sharp upper bound in this case (Section \ref{sec: tables}).
\end{example}

\footnotesize
\bibliographystyle{alpha}
\bibliography{ConfigurationCounts.bbl}\medskip

\noindent Alex Fink: Queen Mary University of London, UK. \href{mailto:a.fink@qmul.ac.uk}{a.fink@qmul.ac.uk} \\
\noindent Navid Nabijou: Queen Mary University of London, UK. \href{mailto:n.nabijou@qmul.ac.uk}{n.nabijou@qmul.ac.uk}\\
\noindent Rob Silversmith: Emory University, USA. \href{mailto:rob.silversmith@emory.edu}{rob.silversmith@emory.edu} 

\end{document}